\documentclass[11pt]{article}
\usepackage{amsmath}
\usepackage{amssymb}
\usepackage[utf8]{inputenc}
\usepackage{float}
\usepackage{mathtools}
\usepackage{mathrsfs}
\usepackage{multicol} 
\usepackage{multirow,array} 
\usepackage{hhline}
\usepackage{float}
\usepackage{yfonts}
\usepackage{polynom}
\usepackage{amsthm}
\usepackage[table]{xcolor}
\usepackage{hyperref}
\usepackage{dsfont}

\usepackage{tikz}
\usepackage{tikz-cd}
\usetikzlibrary{decorations.pathmorphing}
\title{Coming down from infinity for coordinated particle systems}
\author{Varun Sreedhar\thanks{Department of Mathematics, University of California San Diego, \href{mailto:varun_sreedhar@brown.edu}{varun\_sreedhar@brown.edu}}}
\date{\today}
\newtheorem{theorem}{Theorem}
\newtheorem*{theorem*}{Theorem}
\newtheorem{proposition}{Proposition}
\newtheorem{lemma}{Lemma}
\newtheorem{corollary}[lemma]{Corollary}

\theoremstyle{definition}
\newtheorem{defi}{Definition}

\newtheorem{example}{Example}

\newcommand{\R}{\mathbb{R}}

\newcommand{\N}{\mathbb{N}}

\newcommand{\Beta}{\operatorname{Beta}}

\newcommand{\thmref}[1]{Theorem~\ref{T:#1}}
\newcommand{\lemref}[1]{Lemma~\ref{L:#1}}
\newcommand{\corref}[1]{Corollary~\ref{C:#1}}
\newcommand{\defref}[1]{Definition~\ref{D:#1}}

\newcommand{\propref}[1]{Proposition~\ref{P:#1}}

\begin{document}


\maketitle

\begin{abstract}
    We study coming down from infinity for coordinated particle systems. In a coordinated particle system, particles live on a set of sites $V$ and are able to coalesce, migrate, reproduce, and die. The dynamics of these events are coordinated in that many particles may undergo the same action simultaneously. Coming down from infinity is the phenomenon where a process starting with infinitely many particles will almost surely have only finitely many particles after any positive time. This phenomenon can be observed in some $\Lambda$-coalescents, and we give sufficient conditions to observe coming down from infinity in coordinated particle systems.
\end{abstract}

\section{Introduction} 

\noindent In this paper we study coordinated particle systems, introduced in \cite{CPS2021}, which are a family of models where particles which live on a set of sites undergo migration, coalescence, death, and reproduction. These actions are coordinated in the sense that the process may undergo large migration, coalescence, death, or reproduction events where a positive proportion of the particles at a site all take the same action simultaneously. For example a positive proportion of the particles at one site may migrate from one site to another. It is useful to imagine this process taking place on a graph where the vertices are the sites and the edges between sites correspond to where the total rate of any migrations or reproductions between those two sites is positive. The main goal of this paper is to provide necessary and sufficient conditions for a coordinated particle system to come down from infinity.


\subsection{Background}

In this paper we study the phenomenon of coming down from infinity in coordinated particle systems. Coming down from infinity occurs when a process started with infinitely many particles has only finitely many particles after any positive time. This phenomenon has been well studied in coalescent processes where particles reduce in number through mergers with other particles. A very important class of coalescent processes are the $\Lambda$-coalescents which are stochastic processes that take values in the set of partitions of $\N$. Each element of the partition can be thought of as a particle and coalescence is accomplished by merging elements of the partition which can be thought of as many particles joining together to form a single particle. The $\Lambda$-coalescent is characterized by the property that if we restrict the particles to a subset containing $b$ particles, then the rate at which a particular subset containing $k$ particles coalesces to a single particle is given by
\begin{align*}
    \lambda_{b,k} = \int_0^1 z^{k}(1-z)^{b-k} \frac{\Lambda(dz)}{z^2}.
\end{align*}

\noindent We note that the rate above only depends on the number of particles that coalesce and not on the size of each of the particles. One way to interpret the rate $\lambda_{b,k}$ is by the following. First a proportion $z \in (0,1]$ of the particles is chosen to coalesce by the measure $z^{-2}\Lambda(dz)$. The probability our specific subset of $k$ particles of the $b$ total particles coalesce will be $z^{k}(1-z)^{b-k}$. Thus the total rate comes from integrating $z^{k}(1-z)^{b-k}$ over all possible values of $z$ as we see above. A slight exception is the case $k = 2$ where we can see that the above still occurs, but we additionally have pairwise mergers at the fixed rate $\Lambda(\{0\})$. It has been shown in \cite{pitmanlambda} that a coalescent process is exchangeable (rate of coalescence does not depend on particle size), consistent (any restriction to a subset of particles is Markov), and has almost surely asynchronous mergers (no coalescence events occur at the same time) if and only if the process is a $\Lambda$-coalescent. One special $\Lambda$-coalescent is called the Kingman coalescent where each pair of particles merge at rate $1$. The Kingman coalescent corresponds to $\Lambda(dz) = \delta_0$ where $\delta_0$ is the Dirac measure at $0$.
\medskip

\noindent The notion of coming down from infinity for $\Lambda$-coalescents has been well studied. It was shown in \cite{lamdacomedown} that if $\gamma_b = \sum_{k = 2}^b (k-1)\binom{b}{k} \lambda_{b,k}$ is the total rate of decrease in the number of particles, then the $\Lambda$-coalescent comes down from infinity if and only if
\begin{align*}
    \sum_{b = 2}^\infty \gamma_b^{-1} < \infty.
\end{align*}
\medskip

\noindent In recent years there have been several extensions to the $\Lambda$-coalescent model. Some important examples include the spatial $\Lambda$-coalescent and the seed bank coalescent which were introduced in \cite{spatial} and \cite{seedbank2, seedbank1} respectively. The spatial $\Lambda$-coalescent is a model where particles live on a locally finite graph $G$, each particle performs a simple random walk on the graph, and particles on the same vertex of the graph behave like a $\Lambda$-coalescent. It was shown in \cite{spatial} that on a finite graph the spatial $\Lambda$-coalescent comes down from infinity if and only if the (non-spatial) $\Lambda$-coalescent comes down from infinity. In contrast, it was shown in \cite{Angel2012} that a spatial $\Lambda$-coalescent on an infinite graph will not come down from infinity under any conditions in a phenomenon called global divergence. The seed bank coalescent on the other hand is a coalescent where particles live on two sites $\{a,d\}$ and they are able to migrate between them. In this model the particles at site $a$ behave like a Kingman coalescent while particles at site $d$ undergo no coalescence. The rate of migrations between sites is defined by finite measures $M_{ad}$ and $M_{da}$ on $[0,1]$ where the rate of migration from site $a$ to $d$ for a subset of $k$ particles when there are $b$ particles present is given by
\begin{align*}
    \int_0^1 z^k(1-z)^{b-k} \frac{M_{ad}(dz)}{z}.
\end{align*}

\noindent The rates for migrations from site $d$ to $a$ are given similarly. In addition, it was shown in \cite{seedbank2} the necessary and sufficient conditions on the migration measures for the seed bank coalescent to come down from infinity. 
\medskip

\noindent A model that generalizes both the spatial $\Lambda$-coalescent and the seed bank coalescent is the coordinated particle system which was first introduced in \cite{CPS2021}. In this system, particles live on a set of sites and are able to coalesce, die, migrate, and reproduce. All of these actions are coordinated in the sense that one can have large coalescence, death, migration, or reproduction events where a proportion of the particles at a given site all perform one of these four actions at once. This is the same as how in the $\Lambda$-coalescent a large number of particles may coalesce into a single particle. In fact one can recover a $\Lambda$-coalescent from the coordinated particle system by specifying that the system has only a single site and the only dynamic on the particles is coalescence. Similarly one can recover both the spatial $\Lambda$-coalescent as well as the seed bank coalescents as special cases of the coordinated particle system. We shall note that these are not the only models that are special cases of the coordinated particle system; a much larger list can be found in Section 2 of \cite{CPS2021}. 
\medskip

\noindent We shall also note the connection of the coordinated particle system model to the multitype $\Lambda$-coalescents described in \cite{Johnston2023}. Both of these models consider coalescent processes with different types (sites); however, there are some key differences. The first difference is that multitype $\Lambda$-coalescents do not allow for any death, reproduction, or large migration actions. The second difference is that a multitype $\Lambda$-coalescent allows for large merger events where particles of different types may merge together to form a new particle which is not possible in a coordinated particle system. While these two models do exhibit different behavior, there exist models that are both multitype $\Lambda$-coalescents as well as coordinated particle systems. For example a spatial $\Lambda$-coalescent on a finite set of sites and the seed bank coalescent of \cite{seedbank1} are both a coordinated particle system and a multitype $\Lambda$-coalescent. It is also easy to see that the seed bank with simultaneous switching discussed in \cite{seedbank2} is an example of a coordinated particle system which is not a multitype $\Lambda$-coalescent, and one can construct a multitype $\Lambda$-coalescent that is not a coordinated particle system. Thus the coming down from infinity results in this paper are distinct from those found in \cite{Johnston2023}.

\subsection{Definition of the process}

\noindent A coordinated particle system on a finite or countable set of sites $V$ is a Markov process $(X(t))_{t\geq 0}$ with values in $\mathcal{P} \coloneqq (\N_0 \cup \{\infty\})^V$. In \cite{CPS2021} this process was constructed by first considering a process on $\N_0^V$ where the transition rates from a state $\pi \in \N_0^V$ are given by
    \begin{align}\label{eq:rates}
        \pi \mapsto \begin{cases}
        \pi-ke_v+ke_u & \text{ at rate } \int_{0}^1 \binom{\pi(v)}{k}z^k(1-z)^{\pi(v)-k} \frac{M_{vu}(dz)}{z}\quad\text{ for } 1 \leq k \leq \pi(v)\\
        \pi - ke_v & \text{ at rate } \int_{0}^1 \binom{\pi(v)}{k}z^k(1-z)^{\pi(v)-k} \frac{D_{v}(dz)}{z}\quad\text{ for } 1 \leq k \leq \pi(v)\\
        \pi + ke_u & \text{ at rate } \int_{0}^1 \binom{\pi(v)}{k}z^k(1-z)^{\pi(v)-k} \frac{R_{vu}(dz)}{z}\quad\text{ for } 1 \leq k \leq \pi(v)\\
        \pi - (k-1)e_v & \text{ at rate } \int_{0}^1 \binom{\pi(v)}{k}z^k(1-z)^{\pi(v)-k} \frac{\Lambda_{v}(dz)}{z^2}\quad \text{ for } 2 \leq k \leq \pi(v)
        \end{cases},
    \end{align}

\noindent and then extending the state space to $\mathcal{P} = (\N_0 \cup \{\infty\})^V$ by using a generator argument. In \eqref{eq:rates}, the measures $M_{vu}(dz), D_{v}(dz), R_{vu}(dz)$, and $\Lambda_{v}(dz)$ are finite measures on $[0,1]$ and $e_v \in \N_0^V$ is the vector of all $0$'s except for a $1$ at position $v \in V$. These measures control migration, death, reproduction, and coalescence in much the same way the measure $\Lambda(dz)$ controls coalescence in a $\Lambda$-coalescent. One key thing to note is that during a reproduction event, each particle produces only one offspring. We also note that these transition rates in \eqref{eq:rates} represent the dynamics of the process $(X(t))_{t\geq 0}$ when restricted to a finite subset of particles. To understand the structure of the transition rates let us focus on the migration events. We can see that when there are $\pi(v)$ particles at site $v$ there are $\binom{\pi(v)}{k}$ ways to choose a subset of $k$ particles and similar to the $\Lambda$-coalescent above we have a proportion $z > 0$ of the particles chosen to migrate by the measure $z^{-1}M_{vu}(dz)$ which means that the probability that a specific subset of $k$ particles all migrate from $v$ to $u$ is $z^k(1-z)^{\pi(v) - k}$. We shall also note from the rates is the scaling $z^{-1}$ for migration, death, and reproduction while coalescence is scaled by $z^{-2}$. This corresponds to the fact that migration, death, and reproduction are actions that involve a single particle, while coalescence involves at least a pair of particles. 
\medskip

\noindent Now one important type of construction for models like the $\Lambda$-coalescent, spatial $\Lambda$-coalescent, and the seed bank coalescents is the Poisson point process construction. Poisson point process constructions for $\Lambda$-coalescents and its generalizations are constructions of the processes where the dynamics of the particles are determined by Poisson point processes on $\R_+ \times [0,1]$ where an atom $(t,p)$ describes the time and the proportion of particles that undergo an action. These constructions are useful for coupling the coordinated particle systems to other processes in order to compare them to prove results. In this paper, the first main result is the Poisson point process construction for coordinated particle systems on a finite graph which leads to the following theorem.

\begin{theorem}\label{T:PPPconstruction}
    Fix a finite set of sites $V$, and let $\Lambda_v, D_v, M_{uv}$, and $R_{uv}$ be finite measures on $[0,1]$ for $u,v \in V$. Then there exists a probability space on which for each $\pi \in \mathcal{P}$ there exists a c\`adl\`ag Markov process $(X_\pi(t))_{t \geq 0}$ called the coordinated particle system with the following properties:

    \begin{enumerate}
        \item[(a)] $(X_\pi(t))_{t \geq 0}$ takes values in $\mathcal{P}$ and $X_\pi(0) = \pi$.

        \item[(b)] If $\pi_1, \pi_2 \in \mathcal{P}$ with $\pi_1 \leq \pi_2$, i.e. $\pi_1(v) \leq \pi_2(v)$ for all $v \in V$, then $X_{\pi_1}(t) \leq X_{\pi_2}(t)$ for all $t \geq 0$.

        \item[(c)] If $|\pi| < \infty$, i.e. $\sum_{v \in V} \pi(v) < \infty$, then the transition rates of $(X_\pi(t))_{t \geq 0}$ are given by \eqref{eq:rates}.
    \end{enumerate}
\end{theorem}

\subsection{Coming down from infinity}

\noindent We shall now summarize the main results in this work on coming down from infinity, more specifically for instantaneous coming down from infinity where a coordinated particle system $(X(t))_{t \geq 0}$ satisfies $|X(0)| = \infty$, i.e. $\sum_{v \in V} X_v(0) < \infty$, but $P(|X(t)| < \infty \text{ for all } t> 0) = 1$. Thus, for all the following results we shall be assuming that none of measures controlling the dynamics of the coordinated particle system has an atom at $1$. This is because an atom at $1$ in a measure corresponds to every particle at a site taking an action simultaneously after an exponential waiting time. This can cause coordinated particle systems to come down from infinity not instantaneously, for example, by all of the particles at a single site undergoing a death event.
\medskip

\noindent The first main result is an extension of the coming down from infinity result obtained in Theorem 2.7 of \cite{seedbank2} where the authors obtain necessary and sufficient conditions for the seed bank coalescent with simultaneous switching to come down from infinity. We are able to model this seed bank coalescent as a coordinated particle system $(X(t))_{t \geq 0}$ on two sites $\{u,v\}$ where $\Lambda_u(dz) = \delta_0$ is Kingman, $M_{uv}(dz)$ and $M_{vu}(dz)$ are migration measures, and all other measures are $0$. Now the necessary and sufficient conditions shown in Theorem 2.7 of \cite{seedbank2} for this coalescent to stay infinite started with initial configuration $X_u(0) = \infty$ and $X_v(0) = 0$ was that the random variable $Y \sim (M_{uv}([0,1]))^{-1}M_{uv}(dz)$ satisfies $E[-\log Y] = \infty$ with the convention $-\log 0 = \infty$. This condition corresponds to the migration measure being strong enough for an infinite number of particles to jump from site $u$ to site $v$ before any positive time, essentially ``escaping" the Kingman coalescent. It was also shown in \cite{seedbank2} that if we let $U \sim \operatorname{Uniform}([0,1])$ be independent of $Y$, then the condition $E[-\log Y] = \infty$ is equivalent to
\begin{align*}
    \sum_{n = 1}^\infty \frac{E[(1-UY)^n]}{n} = \infty.
\end{align*}

\noindent In this paper we extend this result to seed bank coalescents where the coalescence measure at site $u$ is no longer Kingman, but rather a measure $\Lambda_u(dz) = f(z)dz$ where $f(z) \sim Bz^{1-\alpha}$ as $z \to 0$, i.e.
\begin{align*}
    \lim_{z \to 0^+}\frac{f(z)}{Bz^{1-\alpha}} = 1
\end{align*}

\noindent for some constants $B > 0$ and $\alpha \in (1,2)$. In particular this family of coalescents includes the Beta coalescents where $\Lambda(dz) = \Beta(2-\alpha, \alpha)$. For these coalescents, we obtain the following theorem for when an infinite number of particles to jump from site $u$ to site $v$ before any positive time, ``escaping" the coalescent.

\begin{theorem}\label{T:betasuffstrong}
    Let $(X(t))_{t \geq 0}$ be a coordinated particle system on two sites $\{u,v\}$ with coalescence according to $\Lambda_u(dz) = f(z)\;dz$ where $f(z) \sim Bz^{1-\alpha}$ as $z \to 0$ with $\alpha \in(1,2)$, death at $u$ according to $D_u(dz)$, and migration from $u$ to $v$ according to a measure $M_{uv}(dz)$. We shall assume that all other measures are $0$. Let $Y \sim M_{uv}([0,1])^{-1}M_{uv}(dz)$ and $U \sim \operatorname{Uniform}([0,1])$ be independent random variables. If the random variable $W \coloneqq UY$ satisfies
    \begin{align*}
        \sum_{n=1}^\infty \frac{E[(1-W)^{n}]}{n^{\alpha-1}} = \infty,
    \end{align*}
    
    \noindent then for any positive time $t > 0$ an infinite number of particles will jump from $u$ to $v$ before time $t$ with probability $1$. On the other hand, if
    \begin{align*}
        \sum_{n=1}^\infty \frac{E[(1-W)^{n}]}{n^{\alpha-1}} < \infty,
    \end{align*}

    \noindent then only finitely many particles will ever migrate from $u$ to $v$ with probability $1$.
\end{theorem}

\noindent It is important to note the similarity between this theorem and Theorem 2.7 of \cite{seedbank2} which matches up with the heuristic that the Kingman coalescent corresponds to $\alpha = 2$. Now because the Kingman coalescent is maximal in the speed of coming down from infinity in the sense of \cite{speed}, we should expect it to be easier for a migration measure to satisfy the conditions of \thmref{betasuffstrong} rather than $E[-\log Y] = \infty$ and this is indeed the case. In Section 3, after the proof of \thmref{betasuffstrong}, we shall give an example of a migration measure $M_{uv}(dz)$ where infinitely many particles are able to jump from site $u$ to $v$ if site $u$ has coalescence according to a Beta coalescent with parameter $\alpha < 2$, but only finitely many particles will migrate from site $u$ to $v$ if site $u$ has coalescence according to a Kingman coalescent.
\medskip

\noindent The two other main results of this paper are conditions for a general coordinated particle system to come down from infinity. In the original work on the coordinated particle system, \cite{CPS2021}, necessary and sufficient conditions were shown for the nested coalescent, a special case of the coordinated particle system, to come down from infinity. The conditions proved were on the moment dual of the nested coalescent. In this paper we wish to prove necessary and sufficient conditions for the coordinated particle system to not come down from infinity and for these conditions to be based on the measures $M_{vu}, D_{v}, R_{vu}$, and $\Lambda_{v}$ that drive the process. Motivated by how infinitely many particles are able to jump out of a site due to a strong migration measure, we make the definition that a migration measure $M_{uv}$ or reproduction measure $R_{uv}(dz)$ is \emph{$\Lambda_u$-strong} if before any positive time an infinite number of particles are able to migrate/reproduce out of site $u$ to site $v$. We make the following definition.

\begin{defi}\label{D:lambdastrong}
    Let $\Lambda(dz)$ be a coalescence measure such that the $\Lambda$-coalescent comes down from infinity. We say that a migration measure $M(dz)$ is \emph{$\Lambda$-strong} if with probability $1$ an infinite number of particles migrate out of a site with coalescence according to $\Lambda(dz)$ before time $t$ for all $t > 0$. More rigorously define a coordinated particle system $(X(t))_{t \geq 0}$ on sites $\{u,v\}$ where $\Lambda_u(dz) = \Lambda(dz)$, $M_{uv}(dz) = M(dz)$, $X_u(0) = \infty$, and $X_v(0) = 0$. Set all other measures to be $0$. Then $M$ is $\Lambda$-strong if $P(X_v(t) = \infty) = 1$ for all $t > 0$. Similarly, we shall say that a reproduction measure $R_{uv}(dz)$ is $\Lambda$-strong if with probability $1$ an infinite number of particles reproduce at a site with coalescence according to $\Lambda(dz)$ before time $t$ for all $t > 0$.
\end{defi}

\noindent This definition allows us to prove sufficient conditions for a coordinated particle system to stay infinite.
\begin{theorem}\label{T:comingdown}
    A coordinated particle system on a finite set of sites $V$ started with infinitely many particles at site $v_1 \in V$ will not come down from infinity if the following conditions are met.

    \begin{enumerate}
        \item[(a)] There exists a site $v_k \in V$ such that the $\Lambda_{v_k}$-coalescent does not come down from infinity.

        \item[(b)] There exist sites $v_2, \ldots, v_{k-1}$ such that for every $i = 1, \ldots, k-1$ either $M_{v_iv_{i+1}}(dz)$ or $R_{v_iv_{i+1}}(dz)$ is $\Lambda_{v_i}$-strong.
    \end{enumerate}
\end{theorem}

\noindent We can even make \thmref{comingdown} into a necessary condition rather than just sufficient granted that the coordinated particle system has no reproduction.

\begin{theorem}\label{T:converse}
    Suppose that a coordinated particle system $(X(t))_{t \geq 0}$ on a finite set of sites $V$ with no reproduction. Suppose $(X(t))_{t \geq 0}$ stays infinite. Then the following three conditions must be met.

    \begin{enumerate}
        \item[(a)] There exists a site $v_1 \in V$ with $X_{v_1}(0) = \infty$.
        
        \item[(b)] There exists a site $v_k \in V$ such that the $\Lambda_{v_k}$-coalescent does not come down from infinity. 

        \item[(c)] There exist sites $v_2, \ldots, v_{k-1}$ such that for every $i = 1, \ldots, k-1$ either $M_{v_iv_{i+1}}(dz)$ or $R_{v_iv_{i+1}}(dz)$ is $\Lambda_{v_i}$-strong.
    \end{enumerate}
\end{theorem}
\medskip

\noindent The rest of the paper is structured as follows. In Section 2 we give a Poisson point process construction of the coordinated particle system and prove \thmref{PPPconstruction}. Then in Section 3 we shall provide some more general conditions as compared to \cite{seedbank2} for an infinite number of particles to escape being evaporated by coalescence through migration by proving \thmref{betasuffstrong}. In order to accomplish this, we also generalize some of the results found in \cite{betasmalltime} on the small time dynamics of the number of particles. Finally, in Section 4 we prove conditions on when a coordinated particle system stays infinite where we make use of arguments similar to those found in \cite{Angel2012} to prove \thmref{comingdown}. We shall also give a partial converse by proving \thmref{converse}.
\medskip

\section{Construction of Coordinated Particle Systems}

One important construction of the $\Lambda$-coalescent that is useful in proofs in the Poisson point process construction. Now we wish to give a similar construction for the coordinated particle system. At a high level, we have Poisson point processes on $[0,\infty) \times [0,1]$ that control the proportion and time at which particles coalesce, migrate, reproduce, and die. The construction we give is based on the constructions given for the $\Lambda$-coalescent, $\Xi$-coalescent, and most notably the spatial $\Lambda$-coalescent which can be found in \cite{pitmanlambda}, \cite{xicoalescent}, and \cite{spatial} respectively. 
\bigskip

\noindent We begin by defining the state space of the coordinated particle system $(X(t))_{t \geq 0}$ on a finite set of sites $V$ to be $\mathcal{P} \coloneqq (\N_0 \cup \{\infty\})^V$. For two vectors $\pi,\pi' \in \mathcal{P}$ we shall say $\pi \leq \pi'$ if and only if $\pi(v) \leq \pi'(v)$ for all $v \in V$. Similarly, for a vector $\pi \in \mathcal{P}$ we shall define $\pi \wedge n \in \mathcal{P}$ by $(\pi \wedge n)(v) \coloneqq \pi(v) \wedge n$ for any $n \in \N$. We shall also let $|\pi|$ for $\pi \in \mathcal{P}$ denote
\begin{align*}
    |\pi| \coloneqq \sum_{v \in V} \pi(v).
\end{align*}

\noindent These definitions will aid us in defining the Poisson point process construction which will entail defining how the process acts when the starting configuration has a maximum of $n$ particles at each site, and then taking $n \to \infty$.
\medskip

\noindent Now finally we shall present some bookkeeping, introduced in \cite{CPS2021}, to keep track of where particles are able to move in the coordinated particle system depending on the migration and reproduction measures. If we have finite measures $M_{uv}$ and $R_{uv}$ on $[0,1]$ for $u,v \in V$ with total masses $m_{uv}$ and $r_{uv}$ respectively, then we can define the undirected edge set
\begin{align*}
    E \coloneqq \{(u,v) \in V \times V \colon u \neq v, \max\{m_{uv}, m_{vu}, r_{uv}, r_{vu}\} > 0\}.
\end{align*}

\noindent We shall call the graph $G = (V, E)$ the \emph{interaction graph} for our coordinated particle system. With this we are finally able to give a proof of \thmref{PPPconstruction} by giving a Poisson point process construction of the coordinated particle system.

\begin{proof}[Proof of \thmref{PPPconstruction}]
    The proof of this construction follows a similar structure to the construction for the spatial $\Lambda$-coalescent which is Theorem 1 in \cite{spatial}. First, decompose the measures $\Lambda_v$ for $v \in V$ by setting $\Lambda_v(\cdot) = c_v\delta_0(\cdot) + \Lambda_v^+(\cdot)$ where $\delta_0(\cdot)$ is the Dirac delta measure at $0$ and $\Lambda_v^+(\cdot)$ is a finite measure on $[0,1]$ with no atom at $0$. Similarly we can decompose the measures $D_v, M_{uv}$, and $R_{uv}$ as $d_v\delta_0(\cdot) + D_v^+(\cdot)$, $m_{uv}\delta_0(\cdot) + M_{uv}^+(\cdot)$, and $r_{uv}\delta_0(\cdot) + R_{uv}^+(\cdot)$ respectively. The atoms at $0$ correspond to the independent actions of the particles and the measures on $(0,1]$ correspond to the coordinated actions of the particles.
    \medskip

    \noindent Once we have decomposed the measures as above we can define the Poisson point processes which drive the actions of the particles. We shall start with all of the coordinated actions. To do this we shall first define $P_z(d\xi)$ to be the law of $\xi = (\xi_p)_{p \in \N}$ which is a random vector of independent $\operatorname{Bernoulli}(z)$ random variables. Now letting $dt$ represent Lebesgue measure we can define the following Poisson point processes.

    \begin{itemize}
        \item Let $C^{(v)}$ with $v \in V$ be a Poisson point process on $\R_+ \times [0,1] \times \{0,1\}^\N$ with intensity measure given by $dt \otimes z^{-2}\Lambda_v^+(dz) \otimes P_z(d\xi)$.

        \item Let $D^{(v)}$ with $v \in V$ be a Poisson point process on $\R_+ \times [0,1] \times \{0,1\}^\N$ with intensity measure given by $dt \otimes z^{-1}D_v^+(dz) \otimes P_z(d\xi)$.
        
        \item Let $M^{(uv)}$ with $u,v \in V$ be a Poisson point process on $\R_+ \times [0,1] \times \{0,1\}^\N$ with intensity measure given by $dt \otimes z^{-1}M_{uv}^+(dz) \otimes P_z(d\xi)$.
        
        \item Let $R^{(uv)}$ with $u,v \in V$ be a Poisson point process on $\R_+ \times [0,1] \times \{0,1\}^\N$ with intensity measure given by $dt \otimes z^{-1}R_{uv}^+(dz) \otimes P_z(d\xi)$.
    \end{itemize}
    
    \noindent Now for the independent actions of the particles we need another set of Poisson point processes which we define as follows.

    \begin{itemize}
        \item Let $C^{(v)}_{s_1,s_2}$ with $v \in V$ and $s_1<s_2 \in \N$ be a Poisson point process on $\R_+$ with intensity measure given by $c_vdt$.

        \item Let $D^{(v)}_s$ with $v \in V$ and $s \in \N$ be a Poisson point process on $\R_+$ with intensity measure given by $d_vdt$.

        \item Let $M^{(uv)}_s$  with $u,v \in V$ and $s \in \N$ be a Poisson point process on $\R_+$ with intensity measure given by $m_{uv}dt$.
        
        \item Let $R^{(uv)}_s$ with $u,v \in V$ and $s \in \N$ be a Poisson point process on $\R_+$ with intensity measure given by $r_{uv}dt$.
    \end{itemize}

    \noindent Now using the Poisson point processes above we can begin to construct the process $(X_\pi(t))_{t \geq 0}$ started from $\pi \in \mathcal{P}$. To do this we shall first define $\pi_n \coloneqq  \pi \wedge n $ which represents the restriction to a subset of the particles. Now we can construct the processes $(X_n(t))_{t\geq 0}$ for every $n$, with initial state given by $X_n(0) = \pi_n$, by defining the interactions at an arrival of an atom in the Poisson point processes defined above. We shall begin with the coordinated interactions of the particles.

    \begin{itemize}
        \item If $(t, z, \xi)$ is an atom in $C^{(v)}$ then we merge all of the particles at site $v$ such that $\xi_p = 1$ into a single particle. To do this we first let $N = (X_n(t-))(v)$. Then we define the quantity $K = \sum_{p = 1}^N \xi_p$ which corresponds to the number of particles at site $v$ that were chosen to participate in the coalescence event. If $K > 0$ then we define
        \begin{align*}
            X_n(t) \coloneqq X_n(t-) - (K - 1)e_v.
        \end{align*}

    \end{itemize}

    \noindent The rest of the interactions follow a similar pattern where on arrival of an atom in the correct Poisson point process we flip a coin for each of the particles present at the corresponding site and perform the correct action, but we shall write them here for completeness.
    
    \begin{itemize}

        \item If $(t, z, \xi)$ is an atom in $D^{(v)}$ then we remove all of the particles at site $v$ such that $\xi_p = 1$. Letting $N = (X_n(t-))(v)$ and $K = \sum_{p = 1}^N \xi_p$, we can define
        \begin{align*}
            X_n(t) \coloneqq X_n(t-) - Ke_v.
        \end{align*}

        \item If $(t, z, \xi)$ is an atom in $M^{(uv)}$ then all of the particles at site $u$ such that $\xi_p = 1$ migrate to site $v$. Letting $N = (X_n(t-))(u)$ and $K = \sum_{p = 1}^N \xi_p$, we can define
        \begin{align*}
             X_n(t) \coloneqq X_n(t-) - Ke_u + Ke_v.
        \end{align*}

        \item If $(t, z, \xi)$ is an atom in $R^{(uv)}$ then all of the particles at site $u$ such that $\xi_p = 1$ reproduce and the offspring are placed at site $v$. Letting $N = (X_n(t-))(u)$ and $K = \sum_{p = 1}^N \xi_p$, we can define
        \begin{align*}
            X_n(t) \coloneqq X_n(t-)+ Ke_v.
        \end{align*}
    \end{itemize}

    \noindent Now we also need to define the independent actions of the particles to finish defining the process $(X_n(t))_{t \geq 0}$ which follows again a similar pattern to the coordinated actions described above. We define these independent actions as follows.

    \begin{itemize}
        \item If $t$ is an atom in $C^{(v)}_{s_1,s_2}$ and $\max\{s_1,s_2\} \leq (X_n(t-))(v)$ then these two particles coalesce into a single particle. We can define
        \begin{align*}
            X_n(t) \coloneqq X_n(t-) - e_v.
        \end{align*}

        \item If $t$ is an atom in $D^{(v)}_{s}$ and $s \leq (X_n(t-))(v)$ then we remove the particle from the system representing the death of that particle. We have
        \begin{align*}
            X_n(t) \coloneqq X_n(t-) - e_v.
        \end{align*}

        \item If $t$ is an atom in $M^{(uv)}_{s}$ and $s \leq (X_n(t-))(u)$ then the particle labeled $s$ migrates from site $u$ to site $v$. We define
        \begin{align*}
            X_n(t) \coloneqq X_n(t-) - e_u + e_v.
        \end{align*}

        \item If $t$ is an atom in $R^{(uv)}_{s}$ and $s \leq (X_n(t-))(u)$ then the particle reproduces producing a single offspring and the offspring is placed at site $v$. We define
        \begin{align*}
            X_n(t) \coloneqq X_n(t-) + e_v.
        \end{align*}
    \end{itemize}

    \noindent Now we need to check that this description of transitions makes $(X_n(t))_{t \geq 0}$ a well defined Markov process taking values in $\N_0^V$. We can can see that if there are only $N \in \N$ particles in the system then the rate of points $(t,z,\xi)$ with $\sum_{p=1}^N \xi_p \geq 1$ in each of the Poisson point processes $D^{(v)}, M^{(uv)}, R^{(uv)}$ for $u,v \in V$ is finite. We also see that the rate of points $(t,z,\xi)$ in $C^{(v)}$ with $\sum_{p=1}^N \xi_p \geq 2$ is finite as well. Similarly, the total rate of the independent actions is finite. Now we just need to verify that $(X_n(t))_{t\geq 0}$ almost surely has a finite number of particles for all $t \geq 0$ if $|V| < \infty$. This is accomplished by the exact same stochastic domination argument that appears in Lemma 2.4 of \cite{CPS2021}. We are able to couple the process $(X_n(t))_{t\geq 0}$ to a process $(Y(t))_{t \geq 0}$ on one site which has only reproduction at rate
    \begin{align*}
        R = \sum_{u,v \in V} R_{uv}([0,1]),
    \end{align*}

    \noindent in such a way such that $P(|X_n(t)| \leq Y(t)\;\; \forall t \geq 0) = 1$. Now Example 2.3 of \cite{CPS2021} shows that $E[Y(t)] = e^{Rt} < \infty$ which proves our claim. This makes $(X_n(t))_{t \geq 0}$ a well defined c\`adl\`ag Markov process. We shall also note that this method of construction of $(X_n(t))_{t \geq 0}$ gives us part (c) of \thmref{PPPconstruction} as the rates described in \eqref{eq:rates} are precisely the rates obtained from the Poisson point processes.
    \medskip
    
    \noindent Now to finish proving part (a) of \thmref{PPPconstruction}, we note that by construction using the Poisson point processes, the processes $(X_n(t))_{t \geq 0}$ are consistent in the sense that if $n \leq m$ then $X_n(t) \leq X_m(t)$ for all $t \geq 0$. This is because at time $0$, we have $X_n(0) \leq X_m(0)$ and at the time $t^*$ of any coalescence, death, or migration event $(X_m(t))_{t\geq 0}$ can lose at most the same number of particles at the site $v$ as $(X_n(t))_{t\geq 0}$ plus possibly $(X_m(t^*-))(v) - (X_n(t^*-))(v)$ more. Similarly, at any reproduction event $(X_m(t))_{t\geq 0}$ will add more particles than $(X_n(t))_{t\geq 0}$. This thus means that $X_n(t) \leq X_m(t)$ for all $t \geq 0$. Since these are all consistent we can define $(X_\pi(t))_{t \geq 0}$ by taking $X_\pi(t) \coloneqq \lim_{n \to \infty} X_n(t)$ which means $(X_\pi(t))_{t \geq 0}$ is a well defined c\`adl\`ag Markov process with $X_\pi(0) = \pi$. This proves part (a) of \thmref{PPPconstruction}.
    \medskip

    \noindent Finally, we shall prove part (b) of \thmref{PPPconstruction}. Suppose that $\pi_1, \pi_2 \in \mathcal{P}$ such that $\pi_1 \leq \pi_2$. This means that for all $n \in \N$ we have $\pi_1 \wedge n \leq \pi_2 \wedge n$. Now by the same argument that we discussed above for the consistency, for all $t \geq 0$ we have $X_{\pi_1 \wedge n}(t) \leq X_{\pi_2 \wedge n}(t)$. This means
    \begin{align*}
        X_{\pi_1}(t) = \lim_{n \to \infty} X_{\pi_1 \wedge n}(t) \leq \lim_{n \to \infty} X_{\pi_2 \wedge n}(t) = X_{\pi_2}(t),
    \end{align*}

    \noindent which shows part (b) as desired.
\end{proof}

\noindent For a coordinated particle system it is important to know what is the rate at which $k$ particles interact when there are a total of $b$ particles at a site $u \in V$. To quantify this, we will make the following definition.

\begin{defi}\label{D:rates}
    For a coordinated particle system that has $b$ particles at site $u \in V$, the rates at which $k$ of these particles coalesce, die, migrate to site $v \in V$, or reproduce at site $v \in V$ are
    \begin{align*}
        \lambda_{b,k}^{(v)} &= \int_{0}^1 z^k(1-z)^{b-k} \frac{\Lambda_v(dz)}{z^2},\\
        d_{b,k}^{(v)} &= \int_{0}^1 z^k(1-z)^{b-k} \frac{D_v(dz)}{z},\\
        m_{b,k}^{(u,v)} &= \int_{0}^1 z^k(1-z)^{b-k} \frac{M_{uv}(dz)}{z},\\
        r_{b,k}^{(u,v)} &= \int_{0}^1 z^k(1-z)^{b-k} \frac{R_{uv}(dz)}{z},
    \end{align*}

    \noindent respectively. Similarly we can write the total rate of mergers at site $v \in V$ when there are $b$ particles present as
    \begin{align*}
        \lambda_b^{(v)} = \sum_{k = 2}^b \binom{b}{k} \lambda_{b,k}^{(v)}.
    \end{align*}

    \noindent Similar formulas hold for the total rates of death, migration, and reproduction and we shall denote these quantities $d_b^{(v)}, m_{b}^{(u,v)},$ and $r_{b}^{(u,v)}$ respectively. Note we may drop the superscript denoting the site if $|V| = 1$ or the site is clear from context.
\end{defi}

\section{Beta Coalescents with Death and Migration}

\noindent In this section we shall work towards proving \thmref{betasuffstrong} which answers the following question. To set up the question we shall first let $(X(t))_{t \geq 0}$ be a coordinated particle system on two sites $\{u,v\}$ which is controlled by the measures $\Lambda_u(dz)$, $D_u(dz)$, and $M_{uv}(dz)$. We shall set all other measures to be $0$. The question now is what are necessary and sufficient conditions on the migration measure for an infinite number of particles to migrate from site $u$ to site $v$. \thmref{betasuffstrong} is able to answer the question when the measure $\Lambda_u(dz)$ has sufficiently regular density near zero, i.e. $\Lambda(dz) = f(z)dz$, where for some $B > 0$ and $\alpha \in (1,2)$ we have
\begin{align*}\label{eq:density}\tag{$\star$}
    f(z) \sim Bz^{1-\alpha}, \qquad \text{ as $z \to 0$}.
\end{align*}

\noindent Note that the beta coalescent where $\Lambda = \Beta(2-\alpha, \alpha)$ satisfies the condition above. Now in order to work towards proving \thmref{betasuffstrong} we need to understand the small time dynamics of the number of particles of $(X(t))_{t \geq 0}$ at site $u$. The number of particles at site $u$ can be modeled as a $\Lambda$-coalescent with death where the death represents the particles that migrate away from site $u$. In this case, the following lemma will show us the behavior of the number of particles at site $u$ which will aid in proving \thmref{betasuffstrong}.

\begin{lemma}\label{L:percenthit}
    Let $\Lambda(dz)$ be a finite measure on $[0,1]$ with density satisfying \emph{(}\ref{eq:density}\emph{)}. Now let $(X(t))_{t \geq 0}$ be a coordinated particle system with a single site, coalescence according to $\Lambda(dz)$, death according to a finite measure $D(dz)$ on $[0,1]$, no reproduction, and an infinite number of particles present at $t = 0$. Then
    \begin{align*}
        \lim_{n \to \infty} P(X(t) = n \text{ for some $t \geq 0$}) = \alpha - 1.
    \end{align*}
\end{lemma}

\noindent We see that this theorem is a slight generalization of Theorem 1.8 in \cite{betasmalltime} where we now consider death along with coalescence. Heuristically, one should expect this theorem to be true because the total rate of death when there are $n$ particles present is dominated by the total rate of coalescence when $n$ is large. This can made rigorous by the following lemma since it is known that for coalescents with density satisfying (\ref{eq:density}) around $0$ we have $\lambda_n = O(n^\alpha)$ which is shown by Lemma 4 in \cite{stochasticflows3}.

\begin{lemma}\label{L:asymptoticdeathrate}
    If $d_n = \sum_{k = 1}^n \binom{n}{k} d_{n,k}$ is the total rate of death when $n$ blocks are present, then for all $\alpha \in (1,2)$ we have $n^{-\alpha}d_n \to 0$ as $n \to \infty$.
\end{lemma}

\begin{proof}
    Using the definitions for $d_n$ and $d_{n,k}$, by the binomial theorem we compute that
    \begin{align*}
        d_n =  \int_{0}^1 \sum_{k = 1}^n \binom{n}{k}z^{k-1}(1-z)^{n-k}\;D(dz) = \int_0^1 z^{-1}(1 - (1-z)^n)\;D(dz).
    \end{align*}

    \noindent Now using the fact that $1 - (1-z)^n \leq nz$ on $[0,1]$ we see that
    \begin{align*}
        \int_0^1 z^{-1}(1 - (1-z)^n)\;D(dz) \leq \int_0^1 \frac{nz}{z}\;D(dz) = nD([0,1]).
    \end{align*}

    \noindent Now since $\alpha > 1$, as $n \to \infty$ we see
    \begin{align*}
        n^{-\alpha}d_n \leq n^{1-\alpha}D([0,1]) \to 0,
    \end{align*}

    \noindent as desired.
\end{proof}

\noindent Now we shall begin to work towards proving \lemref{percenthit}. The structure of this proof is very similar to the proof of Theorem 1.8 in \cite{betasmalltime} where the result was shown for the process with coalesence, but without death. We shall begin by investigating the probability that the coordinated particle system loses $k$ particles given that it currently has $n$ particles. Since the only ways to lose $k$ particles in its next transition are to have a coalescence event involving $k+1$ particles, or to have a death event involving $k$ particles, it is clear that this probability will be given by
\begin{align*}
    \zeta_{n,k} = \frac{\binom{n}{k+1} \lambda_{n,k+1} + \binom{n}{k}d_{n,k}}{\lambda_n + d_n}, \qquad k \leq n.
\end{align*}

\noindent From this definition of $\zeta_{n,k}$ we can observe the following limiting behavior for these probabilities.

\begin{corollary}\label{C:convergeindist}
    When $k \leq n-1$, as $n \to \infty$ we have that
    \begin{align*}
        \lim_{n \to \infty} \zeta_{n,k} = \frac{\alpha\Gamma(k + 1 - \alpha)}{(k+1)!\Gamma(2-\alpha)}.
    \end{align*}

    \noindent On the other hand, when $k = n$ we see $\zeta_{n,k} \to 0$.
\end{corollary}

\begin{proof}
    From Lemma 4 in \cite{stochasticflows3} we have $\lambda_n = O(n^\alpha)$ and when $k \leq n-1$
    \begin{align*}
        \lim_{n \to \infty}\frac{\binom{n}{k+1} \lambda_{n,k+1}}{\lambda_n} = \frac{\alpha\Gamma(k + 1 - \alpha)}{(k+1)!\Gamma(2-\alpha)}.
    \end{align*}

    \noindent The resulting limit follows in this case from an application of \lemref{asymptoticdeathrate}. When $k = n$ the result is immediate by $\binom{n}{k}d_{n,k} \leq d_n$ and \lemref{asymptoticdeathrate}.
\end{proof}

\noindent We shall define this limiting quantity by
\begin{align*}
    \zeta_k \coloneqq \frac{\alpha\Gamma(k + 1 - \alpha)}{(k+1)!\Gamma(2-\alpha)},
\end{align*}

\noindent and as in \cite{betasmalltime} we can define the random variable $\zeta$ such that $P(\zeta = k) = \zeta_k$ and $E[\zeta] = \frac{1}{\alpha - 1}$. We can view $\zeta$ as the step distribution of a random walk on the positive integers which approximates the values that the beta coalescent with death achieves on the way down from infinity. This approximation is a key step to proving \lemref{percenthit} since if $(S_n)_{n = 0}^\infty$ represents a random walk with step distribution $\zeta$ then by an application of the discrete Blackwell's Renewal Theorem we have
\begin{align*}
    \lim_{k \to \infty}P(S_n = k \text{ for some $n$}) = \frac{1}{E[\zeta]} =  \alpha -1,
\end{align*}

\noindent Now to use this approximation to prove the theorem we first need to prove the following lemma that will allow us to bound the rate $\zeta_{n,k}$ for large coalescence and death events, i.e. for all $n$ and $n/2 \leq k \leq n$. We shall quickly note that this lemma is related to Lemma 7.1 in \cite{betasmalltime} where the authors are able to bound the rate of large coalescence events in the case of no death, i.e. $d_n = 0$ and $d_{n,k} = 0$ for all $n$ and $k \leq n$.

\begin{lemma}\label{L:lambdaregularity}
    Assume that the measure $\Lambda(dz)$ that controls coalescence has density $f$ that satisfies $C_1z^{1-\alpha} \leq f(z) \leq C_2z^{1-\alpha}$ for all $z \in (0,1]$ with some constants $C_1, C_2 > 0$. Also assume the death measure $D(dz)$ is supported on $[0, 1/8]$. Then there exists $C$ such that $\zeta_{n,k} \leq Ck^{-1-\alpha}$ for all $n$ and $k$ such that $n/2 \leq k \leq n$.
\end{lemma}

\begin{proof}
    Following the proof of Lemma 7.1 of \cite{betasmalltime} we can first construct an upper bound for $\binom{n}{k}\lambda_{n,k}$ by seeing
    \begin{align*}
        \binom{n}{k}\lambda_{n,k} \leq C_2\binom{n}{k}\int_0^1 z^{k-1-\alpha}(1-z)^{n-k}\;dz = \frac{C_2n!\Gamma(k- \alpha)}{k!\Gamma(n-\alpha+1)}\leq C_3n^\alpha k^{-1-\alpha}.
    \end{align*}

    \noindent We similarly get a lower bound $\binom{n}{k}\lambda_{n,k} \geq C_4n^\alpha k^{-1-\alpha}$ for some other constant which lets us lower bound $\lambda_n + d_n$ by seeing that there exists a constant $C_5$ such that
    \begin{align*}
        \lambda_n + d_n \geq \lambda_n = \sum_{k = 2}^n \binom{n}{k}\lambda_{n,k} \geq \sum_{k = 2}^n C_4n^\alpha k^{-1-\alpha} \geq C_5n^\alpha.
    \end{align*}

    \noindent Now we need to bound the term $\binom{n}{k}d_{n,k}$. Noting that $2k \geq n$ we can see that
    \begin{align*}
        \binom{n}{k}d_{n,k} = \binom{n}{k}\int_0^1 z^{k-1}(1-z)^{n-k}D(dz) \leq 2^n \int_0^{\frac{1}{8}} z^{k-1}D(dz) \leq 2^{2k} 2^{-3k+3}D([0,1]) \leq \frac{8D([0,1])}{2^k}.
    \end{align*}
    
    \noindent Thus 
    \begin{align*}
        \zeta_{n,k} = \frac{\binom{n}{k+1} \lambda_{n,k+1} + \binom{n}{k}d_{n,k}}{\lambda_n + d_n} \leq \frac{C_3n^\alpha k^{-1-\alpha} }{C_5n^\alpha} + \frac{8D([0,1])}{C_5n^{\alpha}2^k} \leq Ck^{-1-\alpha},
    \end{align*}

    \noindent for some constant $C > 0$ which proves the claim.
\end{proof}

\begin{lemma}\label{L:largedeath}
    Let $(X(t))_{t \geq 0}$ be a $\Lambda$-coalescent with death according to $D(dz)$ and $\Lambda(dz)$ satisfying \eqref{eq:density} with parameter $\alpha \in (1,2)$. Now if $X(s) = k+j$, then there exists a constant $C > 0$ such that the probability that the next event the process experiences is a death event that brings it below $k$ particles is bounded above by
    \begin{align*}
        \frac{C}{k^\alpha}\left(1 + \frac{k}{j} \right).
    \end{align*}

    \noindent In particular, the constant $C > 0$ is independent of $k$ and $j$.
\end{lemma}

\begin{proof}
   Note that the total rate of any coalescence or death events occurring when $X(s) = k+j$ is $\lambda_{k+j} +d_{k+j}$. Thus the probability that the next event that occurs is a death event involving at least $j + 1$ particles is
   \begin{align*}
       \sum_{\ell = j+1}^{k+j} \frac{\binom{k+j}{\ell}d_{k+j, \ell}}{\lambda_{k+j} +d_{k+j}} = \frac{1}{\lambda_{k+j} +d_{k+j}}\int_0^1\sum_{\ell = j+1}^{k+j}\binom{k+j}{\ell}z^{\ell-1}(1-z)^{k+j-\ell} D(dz).
   \end{align*}

   \noindent Now letting $B_{k+j,z}$ denote a binomial random variable with parameters $n= k+j$ and $p=z$, we can rewrite the integral as
   \begin{align}\label{eq:integralofbinom}
        \frac{1}{\lambda_{k+j} +d_{k+j}}\int_0^1 P(B_{k+j,z} > j)  z^{-1}D(dz).
   \end{align}

   \noindent To obtain an upper bound for this integral we can apply Markov's inequality to see that
   \begin{align*}
        P(B_{k+j,z} > j) \leq \frac{E[B_{k+j, z}]}{j} = \frac{z(k+j)}{j}.
   \end{align*}

   \noindent Plugging this bound back into \eqref{eq:integralofbinom} we see that
   \begin{align*}
       \frac{1}{\lambda_{k+j} +d_{k+j}}\int_0^1 P(B_{k+j,z} > j)  z^{-1}D(dz) \leq \frac{1}{\lambda_{k+j} +d_{k+j}}\int_0^1\frac{k+j}{j}D(dz) \leq \frac{D([0,1])}{\lambda_{k+j}}\left( \frac{k+j}{j}\right).
   \end{align*}
   
   \noindent Now noting that $\lambda_{k + j} = O((k+j)^\alpha)$ we have shown that there exists a constant $C > 0$ independent of $k$ and $j$ such that 
   \begin{align*}
       \frac{D([0,1])}{\lambda_{k+j}}\left( \frac{k+j}{j}\right) \leq \frac{C}{k^\alpha}\left(1 + \frac{k}{j}\right).
   \end{align*}

   \noindent 
\end{proof}

\noindent Now this bound on the $\zeta_{n,k}$ will allow us to show that with high probability the coordinated particle system $(X(t))_{t \geq 0}$ on the single site with coalescence and death will land in an interval $[a,b]$ for some time $t$ given that the size of the interval is sufficiently large. More precisely, we have the following lemma.

\begin{lemma}\label{L:withinrange}
    Define the event $V_{k, k+\ell}^n = \{X^n(t) \in \{k, k+1, \ldots, k+ \ell\} \text{ for some $t$}\}$ where $(X^n(t))_{t \geq 0}$ is the coalescent with death started with $n$ particles and the coalescent measure $\Lambda(dz)$ has density satisfying \emph{(}\ref{eq:density}\emph{)}. For all $\varepsilon > 0$ there exists $L \in \N$ such that $P(V_{k, k+\ell}^n) \geq 1 - \varepsilon$ if $\ell \geq L$ and $n \geq k + \ell$.
\end{lemma}

\begin{proof}
    The proof of this lemma follows a similar pattern to the proof of Lemma 7.2 in \cite{betasmalltime}. We start with the special case where the measures $\Lambda(dz)$ and $D(dz)$ satisfy the conditions outlined in \lemref{lambdaregularity}. To start we shall define the events $V_k^n = V_{k,k}^n = \{X^n(t) =k \text{ for some $t$}\}$. Now to bound the probability $P((V_{k, k + \ell}^n)^c)$ we can note that the coalescent process $(X(t))_{t \geq 0}$ will jump over the interval $[k, k+\ell]$ in two cases. The first case is the coalescent with death jumps from a level above $2k$ to below level $k$ which means the coalescent has lost at least half of its blocks. We shall call this event $H_1$. The second case is where the coalescent with death lands in the interval $[k,2k]$ before jumping over the interval $[k, k+\ell]$. We call this event $H_2$. Since
    \begin{align*}
        P((V_{k,k+\ell}^n)^c) = P(H_1) + P(H_2),
    \end{align*}

    \noindent it suffices to bound the probabilities of $H_1$ and $H_2$. We shall begin by bounding the probability of $H_1$. To do this we can bound $P(H_1)$ by the probability that for any $j \geq 2k + 1$ the process $(X^n(t))_{t \geq 0}$ loses at least $j - k + 1$ blocks when it has $j$ blocks, an event we shall call $B_{j,j-k+1}$. To do this we note that the probability the process loses exactly $h$ blocks when it has $j$ is $\zeta_{j,h}$. Note that since $h \geq j/2$ for all $h \geq j-k+1$, we may apply \lemref{lambdaregularity} to see that $P(B_{j,j-k+1})$ is bounded above by
    \begin{align*}
        P(B_{j,j-k+1}) = \sum_{h = j-k+1}^j \zeta_{j,h} \leq  \sum_{h = j-k+1}^\infty Ch^{-1-\alpha} \leq \int_{j-k}^\infty Cx^{-1-\alpha}\;dx = \frac{C(j-k)^{-\alpha}}{\alpha}.
    \end{align*}
    
    \noindent Now noting that $H_1$ is the event that any of the $B_{j,j-k+1}$ occur for $j \geq 2k+1$ we obtain the bound following bound where we see that
    \begin{align*}
        P(H_1) \leq \sum_{j = 2k+1}^\infty P(B_{j,j-k+1}) \leq \sum_{j = 2k+1}^\infty  \frac{C(j-k)^{-\alpha}}{\alpha} \leq \frac{C}{\alpha}\int_{2k}^\infty (x - k )^{-\alpha}\;dx = \frac{Ck^{1-\alpha}}{\alpha(\alpha - 1)}.
    \end{align*}
    
    \noindent Now we shall bound $P(H_2)$. In order for $H_2$ to occur we see that $(V_{k, k+\ell + i}^n)^c \cap V^n_{k+\ell + 1+ i}$ must have occurred for some $0 \leq i \leq k-\ell - 1$, i.e. $(X(t))_{t \geq 0}$ jumps from level $k+\ell + 1+ i \in (k+\ell, 2k]$ to below $k$ which meaning it avoids the interval $[k, k+\ell + i]$. Thus
    \begin{align*}
        P(H_2) \leq \sum_{i = 0}^{k-\ell -1} P((V_{k, k+\ell + i}^n)^c \cap V^n_{k+\ell + 1+ i}) \leq \sum_{i = 0}^{k-\ell - 1} P((V_{k, k+\ell + i}^n)^c | V^n_{k+\ell + 1+ i}).
    \end{align*}

    \noindent Now to bound each of the terms in the last sum, note that each of the events $((V_{k, k + j}^n)^c | V^n_{k + j+1})$ means either a coalescence event occurred that involved at least $j +3$ particles, or a death event occurred that involves at least $j + 2$ particles. Now we can see that Lemma 7.1 in \cite{betasmalltime} tells us that the probability of a coalescence event involving at least $j +3$ particles is
    \begin{align*}
        \sum_{m = j+3}^{k+j+1} \frac{\binom{k+j+1}{m}\lambda_{k+j+1, m}}{\lambda_{k+j+1} + d_{k+j+1}} \leq \sum_{m = j+3}^{k+j+1} \frac{\binom{k+j+1}{m}\lambda_{k+j+1, m}}{\lambda_{k+j+1}} \leq \sum_{m = j+3}^{\infty} m^{-1-\alpha} \leq \int_{j +2}^\infty x^{-1-\alpha}\;dx = \frac{(j+2)^{-\alpha}}{\alpha}.
    \end{align*}

    \noindent Now for the death event, we have upper bounded the probability that the number of particles drops from $k+j+1$ to below $k$ due to a death event in \lemref{largedeath} where the bound is
    \begin{align*}
        \frac{C}{k^\alpha}\left(1 + \frac{k}{j+1}\right).
    \end{align*}

    \noindent Combining these two results we see that we can obtain the upper bound
    \begin{align*}
        P((V_{k, k+\ell + i}^n)^c | V^n_{k+\ell + 1+ i}) \leq \frac{(\ell + i+2)^{-\alpha}}{\alpha} + \frac{C}{k^\alpha}\left(1 + \frac{k}{\ell + i+1}\right).
    \end{align*}
    
    \noindent Now we can finally use this bound to bound the term $P(H_2)$ by computing
    \begin{align*}
        P(H_2) &\leq \sum_{i = 0}^{k-\ell - 1}\left( \frac{(\ell + i+2)^{-\alpha}}{\alpha} + \frac{C}{k^\alpha}\left(1 + \frac{k}{\ell + i+1} \right)\right) \\
        &\leq \sum_{i = 0}^\infty\frac{(\ell + i+2)^{-\alpha}}{\alpha} + \frac{C}{k^\alpha}\sum_{i = 1}^k \left(1 + \frac{k}{i} \right)\\
        &\leq \int_{\ell +1}^\infty \frac{x^{-\alpha}}{\alpha}\;dx + \frac{C}{k^\alpha}\sum_{i = 1}^k \left(1 + \frac{k}{i}\right)\\
        &\leq \frac{(\ell + 1)^{1-\alpha}}{\alpha(\alpha-1)} + \frac{C}{k^\alpha}\sum_{i = 1}^k \left(1 + \frac{k}{i}\right).
    \end{align*}

    \noindent Combining this with our bound on $P(H_1)$ we see that we have shown 
    \begin{align*}
        P((V_{k,k+\ell}^n)^c) &\leq P(H_1) + P(H_2) \\
        &\leq \frac{(\ell + 1)^{1-\alpha}}{\alpha(\alpha-1)} + \frac{Ck^{1-\alpha}}{\alpha(\alpha - 1)} + \frac{C}{k^\alpha}\sum_{i = 1}^k \left(1 + \frac{k}{i} \right)\\
        &\leq \frac{(\ell + 1)^{1-\alpha}}{\alpha(\alpha-1)} + \frac{Ck^{1-\alpha}}{\alpha(\alpha - 1)} + \frac{Ck}{k^\alpha} + \frac{Ck\log k}{k^\alpha}.
    \end{align*}

    \noindent Now since $\alpha > 1$, the above goes to $0$ as $k, \ell \to \infty$. Now we shall choose $L \in \N$ such that
    \begin{align*}
        \frac{(C+1)(L/2 +1)^{1-\alpha}}{\alpha(\alpha-1)}  + \frac{C(L/2)}{(L/2)^\alpha} + \frac{C(L/2)\log (L/2)}{(L/2)^\alpha}  < \varepsilon.
    \end{align*}

    \noindent We claim that such an $L$ proves the result in this case. To see this we have that if $k \geq L/2$ and $\ell \geq L/2$ then we have
    \begin{align*}
        P((V^n_{k,k+\ell})^c) \leq P(H_1) + P(H_2) < \varepsilon.
    \end{align*}

    \noindent Now if $k < L/2$ we see that if $\ell \geq L$ then $V_{L/2, L/2 + \ell/2}^n \subseteq V^n_{k,k+\ell}$. By the first case where we had $k \geq L/2$ we know 
    \begin{align*}
        1 - \varepsilon \leq P(V_{L/2, L/2 + \ell/2}^n) \leq P(V^n_{k,k+\ell}).
    \end{align*}

    \noindent This shows the desired claim that for all $\varepsilon > 0$ there exists an $L \in \N$ such that $P(V_{k, k+\ell}^n) \geq 1 - \varepsilon$ if $\ell \geq L$ and $n \geq k + \ell$.
    \medskip

    \noindent The general case where the density of $\Lambda(dz) = f(z)dz$ satisfies \eqref{eq:density}, but does not satisfy the assumptions of \lemref{lambdaregularity} also follows as in the general case of Lemma 7.2 in \cite{betasmalltime}, but we shall repeat the argument here for completeness. Since $f(z) \sim Bx^{1-\alpha}$ as $z \to 0$ there exists $\delta > 0$ and constants $C_1, C_2 > 0$ such that $C_1z^{1-\alpha} \leq f(z) \leq C_2z^{1-\alpha}$ for all $z \in (0, \delta]$. Now we can define the measure $\Lambda_*(dz)$ by
    \begin{align*}
        \Lambda_*(dz) = (f(z)\mathds{1}_{(0, \delta]}(z) + x^{1-\alpha}\mathds{1}_{(\delta, 1]}(z))\;dz.
    \end{align*}

    \noindent We shall now define $(X_*^n(t))$ to be a coordinated particle system started with $n$ particles, coalescence according to $\Lambda_*(dz)$, and death according to $D_*(dz) = \mathds{1}_{[0,1/8]}D(dz)$. Now since the restrictions of $\Lambda(dz)$ and $\Lambda^*(dz)$ to $(0, \delta]$ are the same and similarly the restrictions of $D(dz)$ and $D_*(dz)$ to $[0,1/8]$ are the same, we are able to use Lemma 3.1 in \cite{betasmalltime} to couple the processes such that $X^n(s) = X^n_*(s)$ for all $s < t$ for some random time $t > 0$, noting that the proof of Lemma 3.1 in \cite{betasmalltime} also works for the death measure. Thus there exists a fixed time $u$ such that $P(X^n(s) = X^n_*(s) \text{ for all } s < u) \geq 1 - \varepsilon/3$ for all $n$. We can also see that there exists an $M$ such that $P(X^n(u) < M) > 1 - \varepsilon/3$. Finally, since $\Lambda_*(dz)$ satisfies the conditions of \lemref{lambdaregularity} we can apply the first case to obtain a $K$ such that if $\ell \geq K$ and $n \geq k+ \ell$ then 
    \begin{align*}
        P(X^n_*(t) \in \{k, \ldots, k+ \ell\} \text{ for some } t) \geq 1 - \frac{\varepsilon}{3}.
    \end{align*}

    \noindent Now we see that if $X^n(s) = X^n_*(s)$ for all $s \leq u$ and $X^n(u) \leq \ell$ then $V^n_{k,k+\ell}$ occurs if and only if $X^n_*(t) \in \{k, \ldots, k+ \ell\}$ for some $t$. We thus get the result by taking $L = \max\{M,K\}$.
\end{proof}

\begin{lemma}\label{L:probhitfromfar}
        Define the event $V_{k}^n = \{X^n(t) = k \text{ for some $t$}\}$. Then for all $\varepsilon > 0$ there exists $M,N \in \N$ such that if $k \geq M$ and $n \geq k + N$ then $|P(V^n_k) - (\alpha - 1)| < \varepsilon$.
\end{lemma}

\begin{proof}
    The strategy for this proof is essentially the same as for the proof of Theorem 1.8 in \cite{betasmalltime}. We shall begin by defining the sequence of random variables $(X_i)_{i = 1}^\infty$ to be i.i.d. with the same distribution as $\zeta$. Letting $S_n = \sum_{i = 1}^\infty X_i$, we can define the event $A_k = \{S_n = k \text{ for some $n$}\}$. Now as we previously remarked, by the discrete Blackwell Renewal Theorem we have
    \begin{align*}
        \lim_{k \to \infty} P(A_k) = \alpha -1.
    \end{align*}

    \noindent Now fixing $\varepsilon > 0$ we can first choose $L$ as in \lemref{withinrange} such that if $\ell \geq L$ then $P(V_{k,k+\ell}^n) \geq 1-\varepsilon$ for all $n \geq k + \ell$. Now choose $m$ large such that if $j \geq (m-1)L$ then $|P(A_j) - (\alpha - 1)| < \varepsilon$. We shall define $N \coloneqq mL$. 
    \medskip
    
    \noindent Now we can define the events $W_k^n = V^n_{k + (m-1)L, k + mL}$ and the random variables
    \begin{align*}
        H_k^n = \max\{j \leq k + mL \colon X^n(t) = j \text{ for some $t$}\}.
    \end{align*}

    \noindent We shall also let $(B_k)_{k = 1}^\infty$ be a sequence of independent random variables such that $P(B_k = j) = \zeta_{k,j}$ for $j \leq k$. Now for each nonnegative integer $i \leq L$ we can define the function $F_i(b_0, b_1, \ldots, b_{mL - i})$ as follows. Construct the sequence $(x_k)_{k = 0}^{mL-i}$ inductively by
    \begin{align*}
        x_k = \begin{cases}
            0 &  k= 0\\
            x_{k-1} + b_{x_{k-1}} & x_{k-1} < mL- i\\
            x_{k-1} & \text{otherwise}
        \end{cases}.
    \end{align*}

    \noindent Then $F_i(b_0, \ldots, b_{mL-i}) = \mathds{1}_{\{x_j = mL-i \text{ for some $j$}\}}$. One way to view the function $F_i(b_0, \ldots, b_{mL-i})$ is through the following model. Imagine a set of lily pads labeled by $\N$ with a frog sitting on lily pad $0$. Now if the frog is on lily pad $\alpha$ it will jump forward a distance of $b_\alpha$ to lily pad $\alpha + b_\alpha$. Now $F_i(b_0, \ldots, b_{mL-i})$ is the event that the frog lands exactly on lily pad $mL-i$. Now we can view $B_k$ are the number of particles lost when the coalescent with death has $k$ particles. This means if the coalescent with death has $k + mL - i$ particles then the probability that it eventually has $k$ particles will be given by $E[F_i(B_{k+mL - i}, \ldots, B_{k+1})]$ since $F_i(B_{k+mL - i}, \ldots, B_{k+1})$ is the event the coalescent will lose exactly $mL-i$ particles. Thus we see that
    \begin{align*}
        P(V_k^n \cap W_k^n) = \sum_{i = 0}^L P(H_k^n = k + mL - i) E[F_i(B_{k+mL - i}, \ldots, B_{k+1})].
    \end{align*}

    \noindent We also note that $F_i(X_1, \ldots, X_{mL - i}) = \mathds{1}_{A_{mL-i}}$ by reasoning similarly to how we reasoned about $F_i(B_{k+mL - i}, \ldots, B_{k+1})$. Now by \corref{convergeindist} we know that the $B_k$ are converging in distribution to $\zeta$ for all $j$. Thus we can define the constant $M$ such that if $k \geq M$ we can couple $(B_k)_{k = 1}^\infty$ and $(X_i)_{i = 1}^\infty$ such that the the total variation distance of the random vectors $(B_{k+mL - i}, \ldots, B_{k+1})$ and $(X_1, \ldots, X_{mL - i})$ is at most $\varepsilon$ for all $i \leq L$. Thus for all $k \geq M$ we have that
    \begin{align*}
        |E[F_i(B_{k+mL - i}, \ldots, B_{k+1})] - P(A_{mL - i})| < \varepsilon.
    \end{align*}

    \noindent Now noting that we previously chose $m$ such that if $j \geq (m-1)L$ then $|P(A_j) - (\alpha -1)| < \varepsilon$, the above expression means that
    \begin{align*}
        |E[F_i(B_{k+mL - i}, \ldots, B_{k+1})] - (\alpha - 1)| < 2\varepsilon.
    \end{align*}

    \noindent Now from our definition of $H_k$ we know that $H_k^n \geq k + (m-1)L$ if and only if $W_k^n$ occurs. Thus from our previous estimates, when $k \geq M$ and when $n \geq k + N$ we have the lower bound
    \begin{align*}
        P(V_k^n) &\geq P(V_k^n \cap W_k^n) \\
        &= \sum_{i = 0}^L P(H_k^n = k + mL - i) E[F_i(B_{k+mL - i}, \ldots, B_{k+1})] \\
        &\geq (\alpha - 1 - 2\varepsilon)P(H_k^n \geq k + (m-1)L)\\
        &\geq (\alpha - 1 - 2\varepsilon)(1-\varepsilon).
    \end{align*}

    \noindent Similarly with the same condition on $k$ and $n$ we have the upper bound
    \begin{align*}
        P(V_k^n) &\leq  (1-P(W_k^n)) + P(V_k^n \cap W_k^n) \\
        &\leq \varepsilon + \sum_{i = 0}^L P(H_k^n = k + mL - i) E[F_i(B_{k+mL - i}, \ldots, B_{k+1})]\\
        &\leq \varepsilon + (\alpha - 1 + 2\varepsilon).
    \end{align*}

    \noindent This proves the statement of the lemma.
\end{proof}

\noindent This lemma allows us to give a very quick proof of \lemref{percenthit} which we write below.

\begin{proof}[Proof of \lemref{percenthit}]
    \noindent We first note that $\{X(t) = k \text{ for some } t \geq 0\} = V_k = V_k^\infty$ in the notation of \lemref{probhitfromfar}. Thus by \lemref{probhitfromfar}, for all $\varepsilon > 0$ there exists $M \in \N$ such that if $k \geq M$ then $|P(V_k^\infty) - (\alpha - 1)| < \varepsilon$. Thus
    \begin{align*}
        \lim_{k \to \infty}P(X(t) = k \text{ for some } t \geq 0) = \alpha -1,
    \end{align*}

    \noindent as desired
\end{proof}

\noindent Using this result we are now able to prove \thmref{betasuffstrong} which gives a criterion on the migration measures for when an infinite number of particles are able to leave a site before any positive time. This is similar to the condition obtained in the proof of part (b) to Theorem 2.7 in \cite{seedbank2} which considers that $\alpha = 2$ case, but now we have a similar result for coalescents with $\alpha \in (1,2)$.

\begin{proof}[Proof of \thmref{betasuffstrong}]
    To begin this proof, we shall first define the set of events $\{M_n\}_{n \in \N}$ where $M_n$ is the event that a particle migrates from site $u$ to $v$ at some time $t$ when $X_u(t-) = n$. Now we are able to break this proof into four steps. The first step is the show that the sum 
    \begin{align}\label{eq:theoremsum}
        \sum_{n=1}^\infty \frac{E[(1-W)^{n}]}{n^{\alpha-1}}
    \end{align}

    \noindent in the theorem statement converges/diverges if and only if the sum
    \begin{align}\label{eq:reindexedsum}
        \sum_{n=1}^\infty \frac{E[(1-W)^{n-1}]}{n^{\alpha-1}}
    \end{align}
    
    \noindent converges/diverges where the only difference is whether we have $n$ or $n-1$ in the exponent. The second step consists of showing that $P(M_n \text{ i.o.}) = 1$ when
    \begin{align*}
        \sum_{n=1}^\infty \frac{E[(1-W)^{n-1}]}{n^{\alpha-1}} = \infty.
    \end{align*}

    \noindent The third step is showing $P(M_n \text{ i.o.}) = 0$ when the sum above converges. Finally, we shall show that the number of particles that are able to migrate depends precisely on how many of the events $\{M_n\}_{n \in \N}$ occur.
    \medskip

    \noindent We shall begin with step one which follows from using the limit comparison test. We note that the re-indexed sum \eqref{eq:reindexedsum} can be written as
    \begin{align}\label{eq:reindexedsum2}
        \sum_{n=1}^\infty \frac{E[(1-W)^{n-1}]}{n^{\alpha-1}} = \sum_{n = 0}^\infty \frac{E[(1-W)^n]}{(n+1)^{\alpha-1}}.
    \end{align}

    \noindent Now we can take the limit as $n \to \infty$ for the ratio of the terms in the sums \eqref{eq:theoremsum} and \eqref{eq:reindexedsum2} see that
    \begin{align*}
        \lim_{n \to \infty} \left( \frac{E[(1-W)^n]}{(n+1)^{\alpha-1}}\right)^{-1} \left( \frac{E[(1-W)^n]}{n^{\alpha-1}}\right) = \lim_{n \to \infty} \frac{(n+1)^{\alpha-1}}{n^{\alpha - 1}} = 1.
    \end{align*}

    \noindent Thus by the limit comparison test we have shown the first step which is the equivalence of the convergence/divergence of the sums \eqref{eq:theoremsum} and \eqref{eq:reindexedsum}.
    \medskip
 
    \noindent Next, we will show that $P(M_n \text{ i.o.}) = 1$ when $\sum_{n=1}^\infty \frac{E[(1-W)^{n-1}]}{n^{\alpha-1}} = \infty$. In order to do this we wish to apply the Kochen-Stone Lemma which requires that
    \begin{align}\label{eq:kochencond}
        \sum_{n = 1}^\infty P(M_n) = \infty.
    \end{align}

    \noindent To show this divergence, we shall begin as in the proof of Theorem 2.7 of \cite{seedbank2} where we define the quantities $\gamma(n)$ to be the rate at which migration occurs when there are $n$ particles present at site $u$. We can compute that
    \begin{align*}
        \gamma(n) &= \sum_{i = 1}^n\int_0^1 \binom{n}{i}z^i(1-z)^{n-i} \frac{M_{uv}(dz)}{z}\\
        &= \int_0^1 (1- (1-z)^n)\frac{M_{uv}(dz)}{z}\\
        &= \int_0^1 \int_0^1 n(1-yz)^{n-1}\;dy\,M_{uv}(dz)\\
        &= n M_{uv}([0,1])E[(1-W)^{n-1}],
    \end{align*}

    \noindent where $W$ is the random variable defined in the theorem statement. Now if we define the event $V_n = \{X_u(t) = n \text{ for some } t\}$, then we can compute that
    \begin{align*}
        P(M_n|V_n) = \frac{\gamma(n)}{\lambda_n + d_n + \gamma(n)} = \frac{n M_{uv}([0,1])E[(1-W)^{n-1}]}{\lambda_n + d_n + n M_{uv}([0,1])E[(1-W)^{n-1}]}.
    \end{align*}

    \noindent Now noting that $(X_u(t))_{t \geq 0}$ behaves like a $\Lambda_u$-coalescent with death according to $D_u(dz) + M_{uv}(dz)$, we see that $P(M_n)(P(M_n|V_n))^{-1} = P(V_n) \to \alpha - 1$ by \lemref{percenthit}. Since $\alpha -1 \in (0,\infty)$, by the limit comparison test we have that
    \begin{align*}
        \sum_{n = 1}^\infty P(M_n) = \infty \iff \sum_{n = 1}^\infty P(M_n|V_n) = \infty. 
    \end{align*}

    \noindent We can also note that $\lambda_n = O(n^\alpha)$ by Lemma 4 in \cite{stochasticflows3}, $n^{-\alpha}d_n \to 0$ by \lemref{asymptoticdeathrate}, and $n^{-\alpha}\gamma(n) \to 0$ from the fact $\gamma(n) = n M_{uv}([0,1])E[(1-W)^{n-1}]$. Thus we have that
    \begin{align*}
        \sum_{n = 1}^\infty P(M_n|V_n) = \sum_{n = 1}^\infty \frac{\gamma(n)}{\lambda_n + d_n + \gamma(n)} = \infty \iff \sum_{n=1}^\infty \frac{E[(1-W)^{n-1}]}{n^{\alpha-1}} = \infty.
    \end{align*}

    \noindent From the assumption on $M_{uv}(dz)$ the latter sum diverges which means that $\sum_{n = 1}^\infty P(M_n) = \infty$.
    \medskip
    
    \noindent Now applying \lemref{percenthit}, we obtain a constant $L \in \N$ such that if $k \geq L$ then $|P(V_k) - (\alpha - 1)| < \varepsilon$. We note that $\{M_n \text{ i.o.}\}$ occurs if and only if the event $\{M_n \text{ i.o. for $n \geq L$}\}$ occurs. The sum \eqref{eq:kochencond} being infinite allows us to use a variant of the Kochen-Stone Lemma which can be found as Theorem 2 in \cite{Yan2006} which states that
    \begin{align*}
        P(M_n \text{ i.o. for $n \geq L$}) \geq \limsup_{N \to \infty}\left(\sum_{k = L}^N\sum_{j = k+1}^N P(M_k)P(M_j) \right)\Bigg{/}\left(\sum_{k = L}^N\sum_{j = k+1}^N P(M_j \cap  M_k)\right).
    \end{align*}

    \noindent In order to compute this limit we see that for $j > k$ the probability $P(M_j \cap M_k)$ is given by
    \begin{align*}
        P(M_j \cap M_k) &=P(M_j)P(M_k |M_j) \\
        &= P(M_j)P(V_k |M_j)P(M_k |M_j \cap V_k) \\
        &=P(M_j)P(V_k |M_j)P(M_k |V_k),
    \end{align*}

    \noindent where the last equality follows from the fact that $M_j$ and $M_k$ are conditionally independent given the event $V_k$. Now using the fact that $|P(V_k) - (\alpha - 1)| < \varepsilon$ for $k \geq L$ we can obtain the lower bound
    \begin{align*}
        \frac{\sum_{k = L}^N\sum_{j = k+1}^N P(M_k)P(M_j)}{\sum_{k = L}^N\sum_{j = k+1}^N P(M_j \cap  M_k)} &= \frac{\sum_{k = L}^N\sum_{j = k+1}^N P(M_j)P(V_k)P(M_k|V_k)}{\sum_{k = L}^N\sum_{j = k+1}^N P(M_j)P(V_k |M_j)P(M_k |V_k)} \\
        &\geq \frac{\sum_{k = L}^N\sum_{j = k+1}^N P(M_j)P(V_k)P(M_k|V_k)}{\sum_{k = L}^N\sum_{j = k+1}^N P(M_j)P(M_k |V_k)} \\
        &\geq \frac{\sum_{k = L}^N\sum_{j = k+1}^N P(M_j)(\alpha - 1 - \varepsilon)P(M_k|V_k)}{\sum_{k = L}^N\sum_{j = k+1}^N P(M_j)P(M_k |V_k)} \\
        &= \alpha -1 - \varepsilon.
    \end{align*}
    
    \noindent Using this lower bound in the Kochen-Stone lemma means that we have shown
    \begin{align*}
        P(M_n \text{ i.o.})= P(M_n \text{ i.o. for $n \geq L$}) \geq \limsup_{N \to \infty}\frac{\sum_{k = L}^N\sum_{j = k+1}^N P(M_k)P(M_j)}{\sum_{k = L}^N\sum_{j = k+1}^N P(M_j \cap  M_k)} \geq \alpha -1 -\varepsilon.
    \end{align*}

    \noindent Now taking $\varepsilon \to 0$ we obtain the bound $P(M_n \text{ i.o.}) \geq \alpha -1 > 0$. Now we shall note that
    \begin{align*}
        \{M_n \text{ i.o.}\} \in \bigcap_{s > 0} \sigma(X(t)\; \colon t < s),
    \end{align*}

    \noindent the germ $\sigma$-field for the process $(X(t))_{t \geq 0}$ since for every $s > 0$ we know whether $\{M_n \text{ i.o.}\}$ has occurred by time $s$. Now since $(X(t))_{t \geq 0}$ is a right continuous Feller process when restricted just to the site $u$ we can apply Blumenthal's 0-1 Law, Theorem 2.15 in \cite{revuzandyor}, to conclude that $P(M_n \text{ i.o.}) \in \{0,1\}$ and thus proving the claim as we previously showed that $P(M_n \text{ i.o.}) \geq \alpha -1 > 0$.
    \medskip

    \noindent Now to prove that $P(M_n \text{ i.o.}) = 0$ when $\sum_{n=1}^\infty \frac{E[(1-W)^{n-1}]}{n^{\alpha-1}} < \infty$ we see by the previous discussion that
    \begin{align*}
        \sum_{n=1}^\infty \frac{E[(1-W)^{n-1}]}{n^{\alpha-1}} < \infty  \implies \sum_{n = 1}^\infty M_n < \infty.
    \end{align*}

    \noindent Thus by the First Borel-Cantelli Lemma we see that $P(M_n \text{ i.o}) = 0$.
    \medskip

    \noindent Finally, we need to show that the number of particles that are able to migrate is fully controlled by the number of the events $(M_n)_{n \in \N}$ that occurred. To do this let us define the random variable $\mathcal{M}$ to be the total number of particles that migrate from $u$ to $v$. Since the sequence of random variables $(\mathds{1}_{M_n})_{n \in \N}$ represents whether migration event occurred when there were $n$ particles at site $u$ we have the bounds
    \begin{align*}
        \sum_{n \in \N} \mathds{1}_{M_n} \leq \mathcal{M} \leq \sum_{n \in \N} n\mathds{1}_{M_n}.
    \end{align*}

    \noindent Thus $\mathcal{M}$ being infinite or finite is determined solely by how many of the events occurred in the sequence $(M_n)_{n \in\N}$. Thus we have
    \begin{align*}
        P(\mathcal{M} = \infty) = \begin{cases}
            1 & \text{ if } \sum_{n=1}^\infty \frac{E[(1-W)^{n-1}]}{n^{\alpha-1}} = \infty\\
            0 & \text{ if } \sum_{n=1}^\infty \frac{E[(1-W)^{n-1}]}{n^{\alpha-1}} < \infty
        \end{cases}.
    \end{align*}

    \noindent Thus result of the theorem now follows from the equivalence of the convergence/divergence of the two sums \eqref{eq:theoremsum} and \eqref{eq:reindexedsum} that we proved in step one.
\end{proof}

\noindent Now we shall recall from the discussion in the introduction that Theorem 2.7 in \cite{seedbank2} shows a similar if and only if to the one stated \thmref{betasuffstrong}, but for the Kingman coalescent instead of coalescents with sufficiently regular density. In that article the authors show that there are infinitely many migration events if and only if $E[-\log(Y)] = \infty$ where $Y \sim (M_{uv}([0,1]))^{-1}M_{uv}(dz)$. Now as shown by \cite{seedbank2} this condition is equivalent to the random variable $W \coloneqq UY$ where $U$ is uniform on $[0,1]$ satisfying 
\begin{align*}
    \sum_{n=1}^\infty \frac{E[(1-W)^{n}]}{n} = \infty.
\end{align*}

\noindent This is seen by noticing that the sum above is the Taylor series for $-\log(1-x)$. We shall note that this is equivalent to plugging in $\alpha = 2$ into \thmref{betasuffstrong}. Now it is clear that any migration measure that satisfies the condition $E[-\log(Y)] = \infty$ will satisfy the first condition in \thmref{betasuffstrong} since $\alpha \in (1,2)$. However, we shall now give an example of a migration measure that satisfies the condition of \thmref{betasuffstrong}, but not $E[-\log(Y)] = \infty$.

\begin{example}
    We shall define $M(dz) = z^{-\gamma}\mathds{1}_{[0,1]}\;dz$ where $\gamma \in (0,1)$. This is a finite measure on $[0,1]$ with total mass $C = M([0,1])$. Now let $Y \sim C^{-1}M(dz)$, and we shall now compute the quantity $E[(1-UY)^{n-1}]$ which is present in both of the summability conditions above. In order to do this, since the random variables $U$ and $Y$ are independent we can view this expectation as a double integral where we see
    \begin{align}\label{eq:computationofEW}
        E[(1-UY)^{n-1}] =\!\! \int_0^1\!\!\!\int_0^1 (1- uz)^{n-1}\;du M(dz) =\!\! \int_0^1 \frac{1-(1-z)^{n}}{nz}M(dz) = E\left[\frac{1-(1-Y)^n}{nY}\right].
    \end{align}

    \noindent Now to compute the right hand side we use the definition of $M(dz)$ as well as the substitution $u = nz$ to see that
    \begin{align}\label{eq:changeofvariables}
        E\left[\frac{1-(1-Y)^n}{nY}\right] = \frac{1}{C}\int_0^1 \frac{1-(1-z)^{n}}{nz^{1+\gamma}}dz = \frac{1}{Cn^{1-\gamma}}\int_0^n \frac{1-(1 - \frac{u}{n})^n}{u^{1+\gamma}}\;du.
    \end{align}

    \noindent Now note that for all $n \in \N$ we have 
    \begin{align*}
        0 <\frac{1-(1 - \frac{u}{n})^n}{u^{1+\gamma}} \leq \frac{1 \wedge u}{u^{1+\gamma}},
    \end{align*}
    
    \noindent which is an integrable function on $[0,\infty]$. Thus we may apply the Dominated Convergence Theorem to the integral in the final expression of \eqref{eq:changeofvariables} to see that as $n \to \infty$ 
    \begin{align*}
        \lim_{n \to \infty}\int_0^n \frac{1-(1 - \frac{u}{n})^n}{u^{1+\gamma}}\;du = \int_0^\infty \frac{1 - e^{-u}}{u^{1 + \gamma}}\;du < \infty.
    \end{align*}

    \noindent Now noting that
    \begin{align*}
        \int_0^1\frac{1-(1 - \frac{u}{n})^n}{u^{1+\gamma}}\;du > 0,
    \end{align*}

    \noindent for all $n \in \N$, and
    \begin{align*}
        \int_0^\infty\frac{1 - e^{-u}}{u^{1 + \gamma}}\;du > 0,
    \end{align*}

    \noindent there exist constants $C_1, C_2 > 0$ such that for all $n \in \N$, we have
    \begin{align}\label{eq:boundonEW}
        \frac{C_1}{n^{1-\gamma}} \leq E\left[\frac{1-(1-Y)^n}{nY}\right] \leq \frac{C_2}{n^{1-\gamma}}.
    \end{align}

    \noindent Now we shall note that plugging \eqref{eq:computationofEW} into the summability condition of \thmref{betasuffstrong} results in
    \begin{align*}
        \sum_{n = 1}^\infty \frac{E[(1-W)^{n}]}{n^{\alpha - 1}} = \sum_{n = 1}^\infty \frac{E\left[\frac{1-(1-Y)^{n+1}}{(n+1)Y}\right]}{n^{\alpha - 1}}
    \end{align*}

    \noindent Now we can use \eqref{eq:boundonEW} on the above expression to see that
    \begin{align}\label{eq:betasuffbound}
        \sum_{n = 1}^\infty\frac{C_1}{n^{\alpha-1}(n+1)^{1-\gamma}}\leq \sum_{n = 1}^\infty \frac{E[(1-W)^{n}]}{n^{\alpha - 1}} \leq \sum_{n = 1}^\infty\frac{C_2}{n^{\alpha-1}(n+1)^{1-\gamma}}.
    \end{align}

    \noindent Now by \thmref{betasuffstrong} we can see that if $\alpha - \gamma > 1$ then only finitely many particles will be able to migrate out of a site with coalescence according to a measure $\Lambda(dz)$ whose density satisfies \eqref{eq:density} by the upper bound in \eqref{eq:betasuffbound}. Similarly, if $\alpha - \gamma \leq 1$ then infinitely many particles will be able to migrate out of a site with coalescence according to a measure $\Lambda(dz)$ whose density satisfies \eqref{eq:density} due to the lower bound in \eqref{eq:betasuffbound}.
    \medskip
    
    \noindent In particular, $M(dz) = z^{-\gamma}\;dz$ where $\gamma = \alpha/2$ is an example of a migration measure such that an infinite number of particles escape a $\Lambda$-coalescent where $\Lambda = \Beta(2-\alpha, \alpha)$ as we have that $\alpha - \gamma = \alpha - \alpha/2 \leq 1$, but not for a Kingman coalescent since $2 - \gamma > 1$.
\end{example}

\section{Coordinated Particle Systems on Finite Graphs}

\noindent We shall now give a sufficient condition for a coordinated particle system not to come down from infinity. Coming down from infinity means that at $t = 0$ there are infinitely many particles present in the system, but with probability $1$ for every $t > 0$ there are only finitely many particles present. For this entire section we shall assume that the measures that control coalescence, death, and migration don't have an atom at $1$. In order to discuss when a coordinated particle system comes down from infinity we need to first understand the interaction between migration and coalescence. In order to do this we need the concept of a migration or reproduction measure being $\Lambda$-strong where we recall \defref{lambdastrong} in the introduction.
\medskip

\noindent Now we shall establish some results about when migration and reproduction measures are $\Lambda$-strong and some properties of these measures being $\Lambda$-strong. We shall first remark that \thmref{betasuffstrong} implies that if $\Lambda(dz) = f(z)dz$ where $f(z) \sim Bz^{1-\alpha}$ where $B> 0$ and $\alpha \in (1,2)$ as $z \to 0$ then a migration measure is $\Lambda$-strong if and only if the random variable $W \coloneqq UY$ where $U$ is uniform on $[0,1]$ and independent of $Y \sim M([0,1])^{-1}M(dz)$ satisfies
\begin{align*}
    \sum_{n = 1}^\infty \frac{E[(1-W)^{n}]}{n^{\alpha - 1}} = \infty.
\end{align*}

\noindent In addition, we see that Theorem 4.1 in \cite{seedbank1} and the proof of Theorem 2.7 in \cite{seedbank2} imply that if $\Lambda(dz) = \delta_0(dz)$ is Kingman, then a migration measure $M(dz)$ is $\Lambda$-strong if and only if either $M(\{0\}) > 0$ or $Y \sim M([0,1])^{-1}M(dz)$ satisfies $E[-\log(Y)] = \infty$. In fact we can obtain $M(dz)$ is $\Lambda$-strong with $\Lambda(dz) = C\delta_0(dz)$ for $C >0$ if and only if either $M(\{0\}) > 0$ or $Y \sim M([0,1])^{-1}M(dz)$ satisfies $E[-\log(Y)] = \infty$ by just carrying through the constant in the proofs of Theorem 4.1 in \cite{seedbank1} and Theorem 2.7 in \cite{seedbank2}. Now heuristically these conditions should imply that $M(dz)$ is $\Lambda$-strong for arbitrary $\Lambda(dz)$ since the Kingman coalescent is maximal in speed of coming down from infinity (see Corollary 3 in \cite{speed}). We will see that this is indeed the case and establish this in \propref{arblambdastrong} below.

\begin{proposition}\label{P:arblambdastrong}
    If a migration measure $M(dz)$ satisfies either $M(\{0\}) > 0$ or $E[-\log(Y)] = \infty$ where $Y \sim M([0,1])^{-1}M(dz)$, then $M(dz)$ is $\Lambda$-strong for all coalescent measures $\Lambda(dz)$.
\end{proposition}

\noindent However, before we proceed with a proof of \propref{arblambdastrong}, we shall first begin with some more basic properties of measures being $\Lambda$-strong. It is first important to note that a migration or reproduction measure is either $\Lambda$-strong or is not, i.e. the probability an infinite number of particles migrate/reproduce before any positive time is either $0$ or $1$. This is a direct corollary of Blumenthal's 0-1 law for Feller processes since both of these events are in the germ $\sigma$-field as the $\Lambda$-coalescent comes down from infinity. Now we have the following corollary of the definition that says that if an infinite number of particles arrive at a site for any positive time then either the migration measure or the reproduction measure was $\Lambda$-strong.

\begin{corollary}\label{C:eitherMorRstrong}
    Let $(X(t))_{t \geq 0}$ be a coordinated particle system on sites $\{u,v\}$ coordinated by measures $\Lambda_u(dz)$, $M_{uv}(dz)$, $R_{uv}(dz)$, and all other measures are $0$. Letting $X_u(0) = \infty$ and $X_v(0) = 0$, if $P(X_v(t) = \infty) = 1$ for all $t > 0$ then either $M_{uv}(dz)$ or $R_{uv}(dz)$ is $\Lambda_u$-strong.
\end{corollary}

\begin{proof}
    Suppose both $M_{uv}(dz)$ and $R_{uv}(dz)$ were not $\Lambda_u$-strong. Then if we fix a time $t > 0$, since both $M_{uv}(dz)$ and $R_{uv}(dz)$ are not $\Lambda_u$-strong only finitely many particles enter site $v$ before time $t$ i.e. $X_v(t) < \infty$ a.s. By the fact that $(X_v(t))_{t \geq 0}$ is monotone increasing in $t$ we see that for all $s < t$ we have $X_v(s) < \infty$. Now for $s > t$ we see that since the $\Lambda_u$-coalescent comes down from infinity, $X_u(t) < \infty$ a.s. which means that for all $s > t$ we have that $X_v(s) < \infty$. Thus for all $t > 0$ we have that $P(X_v(t) = \infty) = 0$ which is a contradiction.
\end{proof}

\noindent Now we shall note that the definition of $\Lambda$-strong does not make any reference to the death measure. This is because of the following theorem which essentially states that being $\Lambda$-strong is a property that is independent of the death measure.

\begin{proposition}\label{P:deathdoesntmatter}
    Let $D(dz)$ be an arbitrary measure that controls death with no atom at $1$. If a migration measure $M(dz)$ or reproduction measure $R_{uv}(dz)$ is $\Lambda$-strong then an infinite number of particles migrate or reproduce out of a site with coalescence according to $\Lambda(dz)$ and death according to $D(dz)$ before time $t$ for all $t > 0$.
\end{proposition}

\begin{proof}
    We shall first consider the case of migration. First we shall define the coordinated particle system $(X(t))_{t \geq 0}$ on sites $\{u,v\}$ with measures $\Lambda_u(dz) = \Lambda(dz)$, $M_{uv}(dz) = M(dz)$, $D_u(dz) = D(dz)$, $X_u(0) = \infty$, and $X_v(0) = 0$. Now similarly we can define a coordinated particle system $(\widehat{X}(t))_{t\geq 0}$ on a single site with $\widehat{X}(0) = \infty$ and only death according to $D(dz)$ where we use the same Poisson point process for the death events in $(X(t))_{t \geq 0}$ and $(\widehat{X}(t))_{t\geq 0}$. By Proposition 3.10 in \cite{CPS2021} we see $(\widehat{X}(t))_{t\geq 0}$ will not come down from infinity. Now the main idea of the rest of the proof is that since $(\widehat{X}(t))_{t\geq 0}$ does not come down from infinity and the death Poisson point processes between $(\widehat{X}(t))_{t\geq 0}$ and $(X(t))_{t\geq 0}$ are coupled, there are an infinite number of particles in $(X(t))_{t\geq 0}$ which were never involved in a death event up to time $t = 1$. Thus the dynamics of the particles that were never selected for a death event behave like a coordinated particle system $(Y(t))_{t \geq 0}$ which is the same as $(X(t))_{t \geq 0}$ except with no death at site $u$. The result would then follow from the fact that $M_{uv}(dz)$ is $\Lambda_u$-strong. To make this rigorous we shall introduce the following modification of the Poisson point process construction.
    \medskip

    \noindent In this construction we shall consider a process $(\Pi(t))_{t \geq 0}$ taking values in the power set $\mathcal{P}(\N \times \{r,b\} \times\{u,v\})$ from which we can recover $(X(t))_{t \geq 0}$. In this process each particle is labeled by a positive integer and has a flag for its color and location in addition to its label. In this labeling scheme $r$ represents red and $b$ represents blue. Now let $C^{(u)}, D^{(u)}, M^{(uv)}, C_{s_1,s_2}^{(u)}, D_s^{(u)}$, and $ M_s^{(uv)}$ for $s_1, s_2, s \in \N$ with $s_1 < s_2$ be the same Poisson point processes that control $(X(t))_{t \geq 0}$. The definitions for these processes can be found in the Poisson point process construction \thmref{PPPconstruction}. Setting $\Pi(0) = \N \times \{b\} \times \{u\}$, the dynamics of the process $(\Pi(t))_{t \geq 0}$ can now be defined as follows.

    \begin{itemize}
        \item At an arrival of an atom $(t,z, \xi)$ in $C^{(u)}$ we shall merge all the blue particles at site $u$ with $\xi_i = 1$ into a single blue particle at site $u$ with label $i^*$ where
        \begin{align*}
            i^* = \min \{i \colon \xi_i = 1 \text{ and particle $i$ is blue and is located at site $u$}\}.
        \end{align*}

        \item At an arrival of an atom $(t,z, \xi)$ in $D^{(u)}$ we shall color all the particles at site $u$ with $\xi_i = 1$ red.

        \item At an arrival of an atom $(t,z, \xi)$ in $M^{(uv)}$ we shall migrate all the blue particles at site $u$ with $\xi_i = 1$ to site $v$.
    \end{itemize}

    \noindent We similarly define the independent actions for arrivals of atoms $t$ in in the Poisson point processes $C_{s_1,s_2}^{(u)}, D_s^{(u)},$ and $M_s^{(uv)}$ where we coalesce blue particles $s_1$ and $s_2$ at site $u$, color particle $s$ red, and migrate the particle $s$ to site $v$ if it is colored blue. Now we clearly can see that a process with the same law as $(X(t))_{t \geq 0}$ can be recovered by counting the number of blue particles at each site in $(\Pi(t))_{t \geq 0}$ as the red particles represent particles that have died. Using the process $(\Pi(t))_{t\geq 0}$ we are now able to prove the result. Recall that we noted the process $(\widehat{X}(t))_{t \geq 0}$ does not come down from infinity which means that $\widehat{X}(t) = \infty$ for all $t > 0$. Thus for every $t > 0$ there exists an infinite subset $\mathcal{S}_t \subseteq \N$ such that for all $n \in \mathcal{S}_t$, we have $\xi_n = 0$ for all atoms $(s,z, \xi)$ in $D^{(u)}$ with $s \leq t$. Similarly for all $n \in \mathcal{S}_t$, the Poisson point process $D_n^{(u)}$ will have no atoms in $[0,t]$. This set $\mathcal{S}_t$ represents the particles alive at time $t$. Now let $(Y(s))_{s \in [0,t]}$ be the process which counts the blue particles with labels in $\mathcal{S}_t$ in $(\Pi(t))_{t \geq 0}$. We can recognize that for all $s \in [0,t]$ we have $Y(s) \leq X(s)$ by construction. 
    \medskip
    
    \noindent Now by the definition of $(Y(s))_{s \in [0,t]}$ we can see that this process is a coordinated particle system with $Y(0) = |S_t| = \infty$, and measures defined by $\Lambda_u(dz)$ and $M_{uv}(dz)$. In particular, we see that $(Y(s))_{s \in [0,t]}$ behaves like $(X(t))_{t \geq 0}$ except without death. Now we can apply the fact that $M_{uv}(dz)$ is $\Lambda_u$-strong to see that $Y_v(t) = \infty$. This means that an infinite number of particles would have migrated from site $u$ to $v$ in $(X(t))_{t \geq 0}$ before time $t$ as $X_v(t) \geq Y_v(t) \geq \infty$. Now since $t > 0$ was arbitrary we have thus proven the claim.
    \medskip
    
    \noindent Now we can follow the exact same argument for the reproduction measure $R_{uv}(dz)$ to get the full result.
\end{proof}

\noindent We shall now begin working towards a proof of \propref{arblambdastrong} where we shall require two lemmas. The first lemma compares the number of particles in a coalescent with death to a coalescent without death.

\begin{lemma}\label{L:removedeath}
    Let $(X(t))_{t \geq 0}$ be a $\Lambda$-coalescent with death according to $D(dz)$ where $D(\{0\}) = 0$. Similarly, let $(\widehat{X}(t))_{t \geq 0}$ be a $\Lambda_*$-coalescent where $\Lambda_*(dz) = 2\Lambda(dz) + D(dz)$. If $X(0) = \widehat{X}(0)$, then there exists a coupling of $(X(t))_{t \geq 0}$ and $(\widehat{X}(t))_{t \geq 0}$ and a random time $t^* > 0$ a.s. such that
    \begin{align*}
        X(t)+1 \geq \widehat{X}(t),
    \end{align*}

    \noindent for all $t \leq t^*$.
\end{lemma}

\begin{proof}
    The main tool used to prove this lemma is Lemma 7.2 in \cite{Angel2012}. Lemma 7.2 in \cite{Angel2012} states that as long as the total rate of coalescence of subsets of size at least $k$ of an arbitrary coalescent process is bounded by the total rate of coalescence of subsets of size at least $k$ in a $\Lambda$-coalescent for all $k \geq 2$, then we can couple the two processes such that the $\Lambda$-coalescent has fewer particles for all $t$.
    \medskip

    \noindent We shall begin by proving the result when $X(0) = n$ and when $\Lambda(dz)$ is supported on $[0,1/2]$. We start by defining a coalescent process $(\mathring{X}(t))_{t \geq 0}$ on a set of particles labeled $\{0, 1, \ldots, n\}$ in the following way. First decompose the measure $\Lambda(dz) = c\delta_0 + \Lambda^+(dz)$ where $\Lambda^+(dz)$ has no atom at $0$. Then define the Poisson point processes $\pi(\cdot)$ and $\overline{\pi}(\cdot)$ on $\R_+ \times [0,1]$ to have intensity measures $dt \otimes z^{-2}\Lambda^+(dz)$ and $dt \otimes z^{-1}D(dz)$ respectively. Now we can define the dynamics of $(\mathring{X}(t))_{t \geq 0}$ as follows.

    \begin{itemize}
        \item If $(t,z)$ is an atom in $\pi(\cdot)$, flip a $z$-coin for each particle $i \neq 0$. Then merge all the particles whose coin came up heads into a particle which is labeled with the smallest label of those merged.

        \item If $(t,z)$ is an atom in $\overline{\pi}(\cdot)$, flip a $z$-coin for each particle $i \neq 0$. Now merge all the particles whose came up heads along with the particle labeled $0$ into a particle which will also be labeled $0$.

        \item Each pair of particles $i<j$, with $i \neq 0$, merge to form a particle labeled $i$ at rate $c$.
    \end{itemize}

    \noindent We realize for this construction $(\mathring{X}(t))_{t \geq 0}$ is essentially the process $(X(t))_{t \geq 0}$, except when particles die they are merged with the particle $0$ instead of being removed from the system. It is thus clear that we can couple the two processes such that
    \begin{align*}
        X(t) + 1 = \mathring{X}(t),
    \end{align*}

    \noindent for all $t \geq 0$. Now to apply Lemma 7.2 of \cite{Angel2012} we need to upper bound the total rate that subsets of particles with size at least $k$ are coalescing. To do this we shall define $\overline{\lambda}_{m,S}$ to be the rate at which a subset $S$ of the particles coalesce in $(\mathring{X}(t))_{t \geq 0}$ when there are $m$ particles present. We then see that the total rate of coalescence for subsets of size at least $k$ for $2 \leq k \leq m$ is
    \begin{align}\label{eq:splitrates}
        \sum_{S \colon |S| \geq k} \overline{\lambda}_{m,S} &= \sum_{S \colon |S| \geq k, 0 \in S} \overline{\lambda}_{m,S} + \sum_{S \colon |S| \geq k, 0 \notin S} \overline{\lambda}_{m,S}.
    \end{align}

    \noindent Now when $0 \in S$, $\overline{\lambda}_{m,S}$ corresponds to the rate of a death event that involves $|S|-1$ out of $m-1$ particles. Similarly, when $0 \notin S$ this corresponds to the rate of a coalescence event that involves $|S|$ particles out of the $m-1$ particles. Thus we see that \eqref{eq:splitrates} is equal to
    \begin{align}\label{eq:totalrateS}
        \sum_{\ell = k}^m \binom{m-1}{\ell - 1}\int_0^1 z^{\ell - 1}(1-z)^{m-1 - (\ell - 1)}\frac{D(dz)}{z} + \sum_{\ell = k}^{m-1} \binom{m-1}{\ell}\int_0^1 z^{\ell}(1-z)^{m-1 -\ell}\frac{\Lambda(dz)}{z^2}.
    \end{align}

    \noindent Now using that fact that $\binom{m}{\ell} =\binom{m-1}{\ell - 1} + \binom{m-1}{\ell}$ and the fact that $\Lambda(dz)$ is supported on $[0,1/2]$, we may upper bound the quantity in \eqref{eq:totalrateS} by
    \begin{align}\label{eq:rateinlambda*}
        \sum_{\ell = k}^m \binom{m}{\ell}\int_0^1 z^{\ell}(1-z)^{m- \ell}\frac{D(dz)}{z^2} + \sum_{\ell = k}^{m} \binom{m}{\ell}\int_0^1 z^{\ell}(1-z)^{m -\ell}\frac{2\Lambda(dz)}{z^2}.
    \end{align}

    \noindent We can now note that the quantity in \eqref{eq:rateinlambda*} is precisely the rate of mergers of size at least $k$ in the $\Lambda_*$-coalescent. Thus if we let $\lambda_{m, \ell}$ denote the rate that $\ell$ blocks of $m$ coalesce in a $\Lambda_*$-coalescent we have shown
    \begin{align*}
        \sum_{S \colon |S| \geq k} \overline{\lambda}_{m,S} \leq \sum_{\ell\geq k} \binom{m}{\ell}\lambda_{m,\ell}.
    \end{align*}

    \noindent Thus we may apply Lemma 7.2 in \cite{Angel2012} to couple $(\mathring{X}(t))_{t\geq 0}$ to the $\Lambda_*$-coalescent $(\widehat{X}(t))_{t\geq 0}$ such that $\mathring{X}(t) \geq X(t)$ for all $t$. Thus when $X(0) = n$ we have the desired inequality that
    \begin{align*}
        X(t) +1 = \mathring{X}(t) \geq \widehat{X}(t).
    \end{align*}

    \noindent Now the case for $X(0) = \infty$ follows by taking $n \to \infty$ in the first case. Now finally for the general $\Lambda(dz)$ we see that by Lemma 3.1 in \cite{betasmalltime} there exists a coupling of the $\Lambda$-coalescent $(X(t))_{t\geq 0}$ and the $\Lambda(dz)\mathds{1}_{[0,1/2]}(z)$ coalescent $(X'(t))_{t\geq 0}$  random time $t^* > 0$ such that for all $t \leq t^*$, we have $X(t) = X'(t)$. Thus by the special case when $\Lambda(dz)$ is supported on $[0,1/2]$ we have 
    \begin{align*}
        X(t) + 1 = X'(t) +1 \geq \widehat{X}(t)
    \end{align*}

    \noindent for all $t \leq t^*$ as desired.
\end{proof}

\noindent The second lemma needed to prove \propref{arblambdastrong} compares the the number of particles in an arbitrary $\Lambda$-coalescent to a sped up Kingman coalescent, i.e. mergers occur at some fixed positive rate. 

\begin{lemma}\label{L:kingmancontrol}
    Suppose $(X(t))_{t \geq 0}$ is a $\Lambda$-coalescent and $(\widehat{X}(t))_{t \geq 0}$ is a $4\Lambda([0,1])\delta_0$-coalescent, both of which are started with infinitely many particles. Then there exists a random time $t^* > 0$ a.s. such that for all $s \leq t^*$,
    \begin{align*}
        X(s) > \widehat{X}(s)
    \end{align*}
\end{lemma}

\begin{proof}
    To prove this lemma we first let $c = \Lambda([0,1])$ and let $\varepsilon > 0$ be small enough such that $\frac{4(1-\varepsilon)}{(1+\varepsilon)} > 1$. Then by Corollary 3 in \cite{speed}, there exists a random $t^*_1 > 0$ such that for all $s \leq t_1^*$,
    \begin{align}\label{eq:asineq}
        X(s) \geq \frac{2}{cs}(1-\varepsilon)\quad \text{ a.s.}
    \end{align}

    \noindent Now we define $(\widehat{X}(t))_{t \geq 0}$ to be a Kingman coalescent with mergers at rate $4c$. By Theorem 1 in \cite{speed} we have that $\lim_{t \to 0} \widehat{X}(t) (1/2ct)^{-1} = 1$ a.s. This means that there is a random $t_2^*$ such that for all $s \leq t^*_2$,
    \begin{align*}
        \left|\frac{\widehat{X}(s)}{(2cs)^{-1}} - 1\right| < \varepsilon\quad \text{ a.s.}
    \end{align*}

    \noindent Rearranging the inequality above results in the inequality
    \begin{align*}
        \widehat{X}(s) < \frac{1+\varepsilon}{2cs} \quad \text{ a.s.}
    \end{align*}

    \noindent Combining this inequality with \eqref{eq:asineq}, we get that for all $s \leq t^* = \min\{t_1^*, t_2^*\}$,
    \begin{align*}
        X(s) \geq \frac{2}{cs}(1-\varepsilon) = \frac{4(1-\varepsilon)(1+\varepsilon)}{2cs(1+\varepsilon)} > \frac{4(1-\varepsilon)}{(1+\varepsilon)} \widehat{X}(s) \geq \widehat{X}(s)\quad \text{a.s.}
    \end{align*}

    \noindent where the last inequality uses that $\frac{4(1-\varepsilon)}{(1+\varepsilon)} > 1$. Thus we have proved the claim.
\end{proof}

\noindent We shall note that this lemma can be slightly generalized to two arbitrary $\Lambda$-coalescents where one has a slower speed of coming down from infinity than the other with essentially the same proof. Now as we are able to prove the theorem which follows quickly from the above lemma.

\begin{proof}[Proof of \propref{arblambdastrong}]
    Fix $\Lambda(dz)$ to be an arbitrary coalescent measure. Now define $(X(t))_{t \geq 0}$ be a coordinated particle system on sites $\{u,v\}$ where $\Lambda_u(dz) = \Lambda(dz)$, $M_{uv}(dz) = M(dz)$, $X_u(0) = \infty$, and $X_v(0) = 0$. We shall set all other measures to $0$. Now we have two cases, one where $M_{uv}(dz)$ has an atom at $0$, and one where it does not.
    \medskip

    \noindent For the first case, suppose that $M_{uv}(dz)$ has an atom at $0$. Now decompose the measure $M_{uv}(dz)$ by considering $M_{uv}(dz) = m_{uv}\delta_0 + M_{uv}^+(dz)$ where $M_{uv}^+(dz)$ has no atom at $0$ and $m_{uv} > 0$. We can then see that the number of particles that are able to migrate from site $u$ to site $v$ in $(X(t))_{t \geq 0}$ is lower bounded by the number of particles that are able to migrate in a coordinated particle system where instead of particles migrating due to atoms in the Poisson point process corresponding to $M_{uv}^+(dz)$, they are killed. More precisely, we define the coordinated particle system $(\widehat{X}(t))_{t \geq 0}$ which is controlled by the same measures as $(X(t))_{t \geq 0}$ except we shall set $\widehat{M}_{uv}(dz) = m_{uv}\delta_0$ and $\widehat{D}_u(dz) = M_{uv}^+(dz)$. Now if we use the same Poisson point process for the coordinated death in $(\widehat{X}(t))_{t \geq 0}$ and the coordinated migrations in $(X(t))_{t \geq 0}$ we have
    \begin{align*}
        X_v(t) \geq \widehat{X}_v(t) \quad \text{a.s. for all $t \geq 0$}.
    \end{align*}

    \noindent Now by \propref{deathdoesntmatter} we can see that $\widehat{X}_v(t) = \infty$ for all $t \geq 0$ if the measure $\widehat{M}_{uv}(dz) = m_{uv}\delta_0$ is $\Lambda$-strong. Now the fact that $m_{uv}\delta_0$ is $\Lambda$-strong for any finite measure $\Lambda(dz)$ is a corollary of Lemma 5.3 in \cite{Angel2012} which thus proves this case.
    \medskip
    
    \noindent For the second case, suppose that $M_{uv}(\{0\}) = 0$. We shall begin to prove this case by lower bounding the number of particles at site $u$ with a sped up Kingman coalescent with death. To do this we see $(X_u(t))_{t \geq 0}$ behaves like a $\Lambda$-coalescent with death according to $M_{uv}(dz)$. We thus can apply \lemref{removedeath} to see that there exists a random time $t^*_1 > 0$ such that for all $s \leq t^*_1$ we have
    \begin{align}\label{eq:plusone}
        X_u(s) + 1  \geq \mathring{X}(s),
    \end{align}

    \noindent where $(\mathring{X}(t))_{t \geq 0}$ is a $\Lambda_*$-coalescent as described in \lemref{removedeath}. Now we shall apply \lemref{kingmancontrol} to the process $(\mathring{X}(t))_{t \geq 0}$. This gives us another random time $t^*_2 > 0$ and a $4\Lambda_*([0,1])\delta_0$-coalescent $(X'(t))_{t \geq 0}$ such that for all $s \leq t^*_2$ we have
    \begin{align}\label{eq:strict}
        \mathring{X}(s) > X'(s).
    \end{align}

    \noindent Now noting that $(X_u(t))_{t \geq 0}$ and $(X'(t))_{t \geq 0}$ are integer valued processes, combining the inequalities in \eqref{eq:plusone} and \eqref{eq:strict} result in
    \begin{align*}
        X_u(s) \geq X'(s),
    \end{align*}

    \noindent for all $s \leq t^* = \min\{t_1^*, t_2^*\}$. Finally, we can construct a coordinated particle system $(\widehat{X}(t))_{t \geq 0}$ on two sites $\{u,v\}$ with coalescence at site $u$ controlled by $\widehat{\Lambda}_u(dz) = 4\Lambda_*([0,1])\delta_0$ and migrations from $u$ to $v$ according to $\widehat{M}_{uv}(dz) = M_{uv}(dz)$ such that for all $s \leq t^*$ we have
    \begin{align*}
        X_u(s) \geq X'(s) \geq \widehat{X}_u(s),
    \end{align*}

    \noindent and we can furthermore assume that the Poisson point processes that control migrations in $(X(t))_{t \geq 0}$ and $(\widehat{X}(t))_{t \geq 0}$ are the same. This can be done by using the Poisson point process construction to construct $(\widehat{X}(t))_{t \geq 0}$ out of the $4\Lambda_*([0,1])\delta_0$-coalescent $(X'(t))_{t \geq 0}$ by using the same Poisson point processes for migration as for $(X(t))_{t \geq 0}$. To recap we have thus shown that there exists a random time $t^* > 0$ such that for all $s \leq t^*$, $X_u(s) \geq \widehat{X}_u(s)$ and the migrations for $(X(t))_{t \geq 0}$ and $(\widehat{X}(t))_{t \geq 0}$ are coupled. Now the result that $X_v(t) = \infty$ follows for $t < t^*$ as the migration measure $\widehat{M}_{uv}(dz)$ is $4\Lambda_*([0,1])\delta_0$-strong, and more particles will migrate from $u$ to $v$ in $(X(t))_{t \geq 0}$ than $(\widehat{X}(t))_{t \geq 0}$ as $X_u(s) \geq \widehat{X}_u(s)$ for all $s \leq t^*$ and the migrations are coupled. Now as $(X_v(t))_{t \geq 0}$ is monotone increasing in $t$ we have that $X_v(t) = \infty$ for all $t > 0$ which proves the result. 
\end{proof}

\noindent We shall note that if we just treat a reproduction measure as migration, by the exact same proof as in \propref{arblambdastrong} we get the same result for reproduction, which leads to the following proposition.

\begin{proposition}
    If a reproduction measure $R_{uv}(dz)$ satisfies $R_{uv}(\{0\}) > 0$ or $E[-\log(Y)] = \infty$ where $Y \sim R_{uv}([0,1])^{-1}R_{uv}(dz)$, then $R_{uv}(dz)$ is $\Lambda$-strong for all coalescent measures $\Lambda(dz)$.
\end{proposition}

\noindent With this setup we are now able to work toward the proof of \thmref{comingdown}. We shall note that combining \thmref{comingdown} with \propref{arblambdastrong} or \thmref{betasuffstrong} will give sufficient conditions for an arbitrary coordinated particle system to come down from infinity that are easier to check. The high level idea behind the proof of \thmref{comingdown} is condition (b) allows for an infinite number of particles to move from site $v_1$ to site $v_k$ during any time interval $[0,t]$ with $t > 0$. Then condition (a) of \thmref{comingdown} makes sure that an infinite number of the particles that moved to site $k$ remain there which keeps the process infinite. We shall begin the proof of \thmref{comingdown} with a lemma to show that an infinite number of particles enter site $v_k$. 

\begin{lemma}\label{L:inftymovements}
    Under the assumption (b) of \thmref{comingdown}, we have
    \begin{align*}
        \lim_{n \to \infty}P\left(\sup_{s \in [0,t]} X_{v_k}^n(s) \geq m\right) = 1,
    \end{align*}
    
    \noindent for all $m \in \N$ where $(X^n(t))_{t\geq 0}$ is the coordinated particle system started with $n$ particles at site $v_1$.
\end{lemma}

\begin{proof}
    The proof of this fact is essentially the same as the proof of Lemma 5.3 in \cite{Angel2012}, but we shall repeat it here for completeness. We shall first prove this fact when $v_k$ is a neighbor of the site $v_1$, i.e. $k = 2$. In order to do this case we will first show an infinite number of particles jump from site $v_1$ to site $v_2$ before time $t$ for all $t > 0$.
    \medskip

    \noindent To start we shall first note that condition (b) of \thmref{comingdown} means that either the migration measure $M_{v_1v_2}(dz)$ or the reproduction measure $R_{v_1v_2}(dz)$ is $\Lambda_{v_1}$-strong. Without loss of generality we shall assume that $M_{v_1v_2}(dz)$ is $\Lambda_{v_1}$-strong. Now let $\mathcal{M}^n_t$ be a random variable that counts the number of particles that enter $v_2$ from $v_1$ before time $t$ through migration or reproduction. Now define the process $(\widehat{X}^n(t))_{t \geq 0}$ to be a coordinated particle system on sites $\{v_1, v_2\}$ with $\widehat{X}^n_{v_1}(0) = X^n_{v_1}(0) = n$ which has coalescence $\widehat{\Lambda}_{v_1}(dz) = \Lambda_{v_1}(dz)$, migration $\widehat{M}_{v_{1}v_2}(dz) = M_{v_{1}v_2}(dz)$, and death $\widehat{D}_{v_1}(dz)$ given by
    \begin{align*}
        \widehat{D}_{v_1}(dz) = D_{v_1}(dz) + \sum_{\mathclap{v' \sim v_1, v'\neq v_2}} M_{v_1v'}(dz).
    \end{align*}

    \noindent We enforce all other measures controlling $(\widehat{X}^n(t))_{t \geq 0}$ to be $0$. We can also assume that the two coordinated particle systems use the exact same Poisson point processes. One should understand $(\widehat{X}^n(t))_{t \geq 0}$ to be a restriction of $(X^n(t))_{t \geq 0}$ to sites $\{v_1, v_2\}$. Now if we let $\widehat{\mathcal{M}}^n_t$ be the quantity corresponding to $\mathcal{M}^n_t$ for $(\widehat{X}^n(t))_{t \geq 0}$ then we clearly have by construction that $\mathcal{M}^n_t \geq \widehat{\mathcal{M}}^n_t$. Now since $\widehat{M}_{v_1v_2}(dz) = M_{v_1v_2}(dz)$ is $\Lambda_{v_1}$-strong we have that $\widehat{\mathcal{M}}^n_t \to \infty$ as $n \to \infty$ for all $t > 0$ by \propref{deathdoesntmatter}. Therefore, $\mathcal{M}^n_t \to \infty$ as $n \to \infty$ a.s. 
    \medskip
    
    \noindent Now to prove the claim for the case $k = 2$ we let $\eta > 0$ and use it as well as the rates defined in \defref{rates} to define the quantity $t_0$ by
    \begin{align*}
        t_0 = \eta \min\left\{t, \left({\lambda_m^{(v_k)}}\right)^{-1}, \left({d_m^{(v_k)}}\right)^{-1}, \left(\sum_{u \sim v_k} m_m^{(v_k,u)}\right)^{-1}\right\}.
    \end{align*}
    
    \noindent Now since $\mathcal{M}^n_t \to \infty$ as $n \to \infty$ for all $t > 0$, there exists $n_0$ large such that if $v_1$ has $n_0$ particles then $v_2$ receives at least $m$ particles before time $t_0$ with probability at least $1-\eta$. Now since the total rate that $\ell$ particles migrate, coalesce, or die are $\lambda_\ell^{(v_2)}, d_\ell^{(v_2)},$ and $\sum_{u \sim v_2} m_\ell^{(v_2,u)}$ respectively (all increasing in $\ell$), the probability that any of these particles move out of $v_2$ before time $t_0$ due to these forces is at most $3\eta$. Thus we have that for all $n \geq n_0$
    \begin{align*}
        P\left( X_{v_2}^n(t_0) \geq m\right) \geq 1 - 4\eta,
    \end{align*}

    \noindent which proves the result for $k = 2$ since $\eta$ can be made arbitrarily small.
    \medskip

    \noindent Now to prove the result for $k \geq 3$ we can apply the induction argument of Lemma 5.3 of \cite{Angel2012}. If the distance $d(v_1, v_k) \neq 1$, then taking $v_{k-1} \sim v_k$ in the path defined by condition (b) of \thmref{comingdown}, and we see $d(v_1, v_{k-1}) < d(v_1, v_k)$. Now by the inductive hypothesis, if we fix $m' \in \N$ and $\eta >0$, for sufficiently large $n$ we have
    \begin{align*}
        P\left(\sup_{s \in [0,t/2]} X_{v_{k-1}}^n(s) \geq m'\right) > 1 - \eta.
    \end{align*}

    \noindent Now by the strong Markov property we can repeat the argument for the base case for sufficiently large $m'$ playing the role of $n_0$ in the base case to get that with probability $1- 5\eta$ vertex $v_k$ has at least $m$ particles at some time $s < t$. Thus we have proved the result via induction.
\end{proof}

\noindent We can see that \lemref{inftymovements} guarantees us that an infinite number of particles will enter site $v_k$ before any positive time. Now we wish to employ condition (a) of \thmref{comingdown} to claim that the infinite number of particles that moved to site $v_k$ will be unable to be reduced to a finite number by the coalesence and death at site $v_k$. To make this rigorous we shall prove the following coupling.

\begin{lemma}\label{L:killingcouple}
    Suppose a $\Lambda$-coalescent with death according to $D(dz)$ has an infinite number of particles enter before time $t_0$, i.e. particles enter at random times $(s_n)_{n \in \N}$ where $s_n \leq t_0$ for all $n \in \N$. Denote the block counting process of this coalescent by $(\mathring{N}(t))_{t \geq 0}$. Then there is a coupling such that $N(t_0) \leq \mathring{N}(t_0)$ where $(N(t))_{t\geq 0}$ is the block counting process of the $\Lambda$-coalescent with killing according to $D(dz)$ starting with infinitely many particles.
\end{lemma}

\begin{proof}
    To prove this coupling we shall construct a process $(\Pi(t))_{t \geq 0}$ using Poisson point processes from which we can recover $(\mathring{N}(t))_{t \geq 0}$ and $(N(t))_{t \geq 0}$. The process $(\Pi(t))_{t \geq 0}$ is a colored coalescent with frozen particles and death. We shall first construct $(\Pi(t))_{t \geq 0}$ in the case when there are $n$ particles in the system and then we shall take $n \to \infty$ for the full result. 
    \medskip

    \noindent We shall begin by defining the set of possible labels to be the set $\{1, \ldots, n\} \times \{r,b\} \times\{f, u\}$ where each particle has an additional flag for its color, $r$ for red and $b$ for blue, and a flag that indicates whether the particle is frozen or unfrozen, $f$ for frozen and $u$ for unfrozen. Now we can define the initial state of the process by $\Pi(0) = \{1, \ldots, n\} \times \{b\} \times \{f\}$ where all the particles are colored blue. Now we shall let $(s_i)_{i = 1}^n$ be the entry times of the particles as described in the lemma statement. For the process $(\Pi(t))_{t \geq 0}$ these will represent the times at which a particle will become unfrozen, i.e. at time $s_i$ the particle labeled $i$ will have its flag flipped from $f$ to $u$. Now we shall define the dynamics of the system such that the number of blue particles at time $t_0$ is $N(t_0)$ and the total number of particles at time $t_0$ is $\mathring{N}(t_0)$.
    \medskip

    \noindent To define the dynamics of $(\Pi(t))_{t \geq 0}$ we shall use Poisson point processes. In order to do this we shall first split the measures $\Lambda(dz)$ and $D(dz)$ by considering $\Lambda(dz) = c \delta_0 + \Lambda^+(dz)$ and $D(dz) = d \delta_0 + D^+(dz)$ such that $\Lambda^+(dz)$ and $D^+(dz)$ do not have atoms at $0$. We shall also define $P_z(d\xi)$ to be the law of $\xi = (\xi_i)_{i = 1}^n$ which is a random vector of independent $\operatorname{Bernoulli}(z)$ random variables. This lets us define the following Poisson point processes:

    \begin{itemize}
        \item Let $\pi(\cdot)$ be a Poisson point process on $\R_+ \times [0,1] \times \{0,1\}^n$ with intensity measure given by $dt \otimes z^{-2}\Lambda^+(dz) \otimes P_z(d\xi)$.

        \item Let $\overline{\pi}(\cdot)$ be a Poisson point process on $\R_+ \times [0,1] \times \{0,1\}^n$ with intensity measure given by $dt \otimes z^{-1}D^+(dz) \otimes P_z(d\xi)$.

        \item Let $\pi_{i,j}(\cdot)$ for $i < j \in \{1, \ldots, n\}$ be a Poisson point process on $\R^+$ with intensity $c\, dt$.
        
        \item Let $\overline{\pi}_{i}(\cdot)$ for $i \in \{1, \ldots, n\}$ be a Poisson point process on $\R^+$ with intensity $d\, dt$.
    \end{itemize}
    
    \noindent Through the use of the Poisson point processes above, the dynamics of the process $(\Pi(t))_{t \geq 0}$ will be defined as follows. When $(t,z, \xi)$ is an atom in $\pi(\cdot)$, we coalesce all the particles in $(\Pi(t))_{t \geq 0}$ with $\xi_i = 1$ that are unfrozen at time $t$ which means $s_i < t$. The resulting particle is labeled with the smallest label $i^* = \min\{i \colon \xi_i = 1\}$. For its color we shall color it blue if any blue particles were coalesced and red if only red particles were coalesced. In addition, for all $i > i^*$ such that $\xi_i = 1$, but were frozen at the time $t$, i.e. $s_i > t$ we shall color these particles red. 
    \medskip
    
    \noindent Similarly, upon an arrival of the atom $(\overline{t}, \overline{z}, \xi)$ in $\overline{\pi}(\cdot)$ we kill all the particles that are unfrozen in $(\Pi(t))_{t \geq 0}$ such that $\xi_i = 1$ and color the frozen particles with $\xi_i = 1$ red. 
    \medskip
    
    \noindent For the independent coalesence and death events the dynamics are similar. When $t$ is an arrival of an atom in $\pi_{i,j}(\cdot)$ we shall coalesce particles labeled $i$ and $j$ in $(\mathring{\Pi}(t))_{t \geq 0}$ if they are both unfrozen into a particle labeled $i$ which is colored red if both $i$ and $j$ are colored red and blue otherwise. If either $i$ or $j$ is frozen we shall perform no coalescence and color particle $j$ red instead. Similarly, when $t$ is an atom in $\overline{\pi}_{i}(\cdot)$ we shall kill the particle labeled $i$ in $(\Pi(t))_{t \geq 0}$ if it is unfrozen and color it red otherwise. 
    \medskip

    \noindent Now we are able to recover that $N(t_0)$, the block counting process for a coalescent with death, is the number of blue particles in $\Pi(t_0)$ and $\mathring{N}(t_0)$ is the total number of particles in $\Pi(t_0)$. It is immediate from this construction that $N(t_0) \leq \mathring{N}(t_0)$. Now taking the starting number of particles $n \to \infty$ gives the desired result.
\end{proof}

\begin{proof}[Proof of \thmref{comingdown}]
    We begin by noting that since $X_{v_1}(0) = \infty$ and condition (b) of the theorem holds we can apply \lemref{inftymovements} to see that for all $m \in \N$,
    \begin{align*}
        P\left(\sup_{s \in [0,t]} X_{v_k}(s) \geq m\right) = \lim_{n \to \infty}P\left(\sup_{s \in [0,t]} X_{v_k}^n(s) \geq m\right) = 1.
    \end{align*}

    \noindent Thus $P(\sup_{s \in [0,t]} X_{v_k}(s) = \infty) = 1$ for every $t > 0$. This can only happen if either there is an infinite number of particles entering site $v_k$ before time $t$ or if site $v_k$ originally started with infinitely many particles. In either case we can use \lemref{killingcouple} to see that $X_k(t) \geq N(t)$ where $(N(t))_{t \geq 0}$ is the block counting process of a single site coalescent with death whose dynamics are coalescence with respect to $\Lambda_{v_k}$ and death with respect to $D_{v_k} + \sum_{u \sim v_k} M_{v_ku}$. Now we may either apply Lemma 4.1 in \cite{Johnston2023}, which states that a $\Lambda$-coalescent with death comes down from infinity if and only if the $\Lambda$-coalescent does, or use a strategy similar to the proof of \propref{deathdoesntmatter} to conclude that $N(t) = \infty$ since the $\Lambda_{v_k}$-coalescent does not come down from infinity by condition (a) of \thmref{comingdown}. Thus the coordinated particle system does not come down from infinity. 
\end{proof}

\noindent Now we shall show a partial converse to \thmref{comingdown} which will be applicable when there is no reproduction. In order to do this we shall need the following lemma.

\begin{proposition}\label{P:comedown}
    Let $(X(t))_{t \geq 0}$ be a coordinated particle system on a finite set of sites $V$ with no death and no reproduction which starts with an infinite number of particles at each site. If for each $v \in V$ we have that the $\Lambda_v$-coalescent comes down from infinity, then $(X(t))_{t \geq 0}$ comes down from infinity.
\end{proposition}

\noindent The strategy to prove this lemma is the same as the strategy to prove the reverse direction of Theorem 1.8 in \cite{Johnston2023} which proves necessary conditions for multitype $\Lambda$-coalescents to come down from infinity. Due to this we shall use similar notation throughout the argument. Before we begin let us recall some facts about the $\Lambda$-coalescent. The first fact can be found in \cite{stochasticflows3} where it was shown that the summability condition found in \cite{lamdacomedown} for a $\Lambda$-coalescent to come down from infinity is equivalent to the fact that for every $s \geq 2$ we have
\begin{align*}
    \int_s^\infty \frac{1}{\psi(q)}\;dq < \infty,
\end{align*}

\noindent where $\psi(q)$ is the \emph{non-asymptotic processing speed} defined by
\begin{align*}
    \psi(q) \coloneqq \Lambda(\{0\})\frac{q(q-1)}{2}+\int_{(0,1]} \Big(qz - 1 + (1-z)^q\Big)\;\frac{\Lambda(dz)}{z^2}.
\end{align*}

\noindent The second fact is that if $\gamma_b$ is the rate at which blocks decrease in the $\Lambda$-coalescent when there are $b$ blocks present we have by an application of the binomial theorem that
\begin{align*}
    \gamma_b = \sum_{k = 2}^b (k-1)\binom{b}{k} \lambda_{b,k} = \Lambda(\{0\})\frac{b(b-1)}{2} + \sum_{k = 2}^b (k-1)\binom{b}{k} \int_{(0,1]} z^k(1-z)^{b-k}\frac{\Lambda(dz)}{z^2} = \psi(b).
\end{align*}

\noindent With these preliminaries we shall begin with the analogue of Lemma 4.2 in \cite{Johnston2023}.

\begin{lemma}\label{L:expchange}
    Suppose a coordinated particle system $(X(t))_{t \geq 0}$ on a finite set of sites $V$ has no death and no reproduction. Assuming that $|X(t)| < \infty$, the expected rate of change in total number of particles at time $t$ is given by
    \begin{align*}
         \lim_{h \downarrow 0} \frac{1}{h}E\left[\sum_{v\in V} X_v(t+h) - X_v(t) \Big| X(t)\right]=-\sum_{v \in V} \psi_v(X_v(t)),
    \end{align*}

    \noindent where $\psi_v$ is the non-asymptotic processing speed of the $\Lambda_v$-coalescent.
\end{lemma}

\begin{proof}
    We first see that since migrations do not change the total number of particles and the fact there is no death or reproduction, the only thing that may cause a change in total number of particles is coalescence. Let $\gamma^{(v)}_b$ be the rate at which blocks decease in the $\Lambda_v$-coalescent when there are $b$ blocks present. Now as $|X(t)| < \infty$ we see that the rate at which particles are lost due to coalescence at time $t$ will be given by
    \begin{align*}
        \sum_{v \in V} \gamma^{(v)}_{X_v(t)} = \sum_{v \in V} \psi_v(X_v(t)),
    \end{align*}

    \noindent which thus proves the result.
\end{proof}

\noindent Now we wish to find lower bounds on the rate of decrease of particles. As was done in \cite{Johnston2023}, if $V = \{v_1, \ldots, v_N\}$ we shall define the function $\Omega \colon [0,\infty) \to [0, \infty)$ by 
\begin{align*}
    \Omega(x) = \min_{x_1 + \cdots + x_N = x} \sum_{i = 1}^N \psi_{v_i}(x_i).
\end{align*}

\noindent To apply Jensen's inequality later on we have the following lemma which is an analogue of Lemma 4.3 in \cite{Johnston2023} for our situation.

\begin{lemma}\label{L:convex}
    The functions $\Omega\colon [0,\infty) \to [0, \infty)$ and $\Psi(x_1, \ldots, x_N) = \sum_{i = 1}^N \psi_{v_i}(x_i)$ from $\R_{\geq 0}^V \to \R$ are convex.
\end{lemma}

\begin{proof}
    The proof of this lemma is essentially the same as Lemma 4.3 in \cite{Johnston2023}. We shall first show that $\Psi(x_1, \ldots, x_N)$ is convex. We see that for each fixed $z \in [0,1]$ and $i$, the map
    \begin{align*}
        x \mapsto x_iz - 1 + (1-z)^{x_i},
    \end{align*}

    \noindent is convex. Thus we have that for each $i$ that the integral
    \begin{align*}
        x \mapsto \int_{(0,1]} x_iz - 1 + (1-z)^{x_i} \frac{\Lambda_{v_i}(dz)}{z^2},
    \end{align*}

    \noindent is also convex. Thus from the definition of $\psi_{v_i}$ we have that $\sum_{i = 1}^N \psi_{v_i}$ is a sum of convex functions which means that $\Psi(x_1, \ldots, x_N)$ is thus convex as desired. Now by the exact same computation as in Lemma 4.3 of \cite{Johnston2023} we can show that $\Omega$ is convex by seeing that
    \begin{align*}
        \Omega(\lambda x + (1-\lambda)y) &= \min_{z_1 + \cdots + z_N = \lambda x + (1-\lambda)y} \Psi(z_1, \ldots, z_N)\\
        &= \min_{x_1 + \cdots + x_N = x,\;y_1 + \cdots +y_N = y} \Psi(\lambda x_1 + (1-\lambda)y_1, \ldots, \lambda x_N + (1-\lambda)y_N)\\
        &\leq \min_{x_1 + \cdots + x_N = x,\;y_1 + \cdots +y_N = y} \lambda\Psi( x_1 , \ldots,x_N) + (1-\lambda)\Psi( y_1 , \ldots,y_N)\\
        &= \lambda \Omega(x) + (1-\lambda)\Omega(y).
    \end{align*}

    \noindent Thus $\Omega$ is convex.
\end{proof}

\noindent Now we have the following lemma which extends the integrability condition of the processing speed to the rate $\Omega$.

\begin{lemma}\label{L:upperboundminrate}
    Let $s>0$. Suppose that for all $i = 1, \ldots, N$ we have $\int_{s/N}^\infty \frac{1}{\psi_{v_i}(q)}\;dq$ converges. Then $\int_{s}^\infty \frac{1}{\Omega(q)}\;dq < \infty$.
\end{lemma}

\begin{proof}
    The proof of this lemma is essentially the same as the proof of Lemma 4.6 in \cite{Johnston2023}. We can first see that if $x_1 + \cdots + x_N = x$ then there exists $j$ such that $x_j \geq x/N$. Noting that $\psi_v$ is an increasing function for all $v \in V$, we have 
    \begin{align*}
        \Omega(x) = \min_{x_1 + \cdots + x_N = x} \sum_{i = 1}^N \psi_{v_i}(x_i) \geq \min_{j = 1, \ldots, N} \psi_{v_j}(x/N).
    \end{align*}

    \noindent We thus can obtain the following upper bound
    \begin{align*}
        \frac{1}{\Omega(x)} \leq \frac{1}{\min_{j = 1, \ldots, N} \psi_{v_j}(x/N)} = \max_{j = 1,\ldots, N} \frac{1}{\psi_{v_j}(x/N)}\leq \sum_{i = 1}^N \frac{1}{\psi_{v_i}(x/N)}.
    \end{align*}

    \noindent Thus thus gives us our desired claim as we see that
    \begin{align*}
        \int_{s}^\infty \frac{1}{\Omega(q)}\;dq \leq \sum_{i = 1}^N \int_{s}^\infty \frac{1}{\psi_{v_i}(q/N)}\;dq = N\sum_{i = 1}^N \int_{s/N}^\infty \frac{1}{\psi_{v_i}(q)}\;dq < \infty.
    \end{align*}
\end{proof}

\noindent Now we can use the above lemma to control the total number of blocks in the system via an ordinary differential equation comparison argument. This lemma is an analogue of Lemma 4.7 in \cite{Johnston2023}.

\begin{lemma}\label{L:odecomp}
    Let $(X_n(t))_{t \geq 0}$ be a coordinated particle system on a finite set of sites $V = \{v_1, \ldots, v_N\}$ with no death and no reproduction. The initial state will be given by $X_n(0) = (n, \ldots, n)$ where there are $n$ particles at each site. We shall now define $f_n(t) = E[|X_n(t)|]$ to be the expected total number of blocks at time $t$. Define $w_n(t)$ to be the solution to the integral equation 
    \begin{align*}
        w_n(t) = nN - \int_0^t \Omega(w_n(s))\;ds.
    \end{align*}

    \noindent Then $f_n(t) \leq w_n(t)$ for all $t \geq 0$.
\end{lemma}

\begin{proof}
    The proof of this lemma is essentially the same of the proof of Lemma 4.7 in \cite{Johnston2023}, but we shall repeat it for completion. We first see by \lemref{expchange} that 
    \begin{align*}
        f_n(t) = nN - \int_0^t E[\Psi(X_{v_1}(s), \ldots, X_{v_N}(s))]\;ds,
    \end{align*}

    \noindent where $\Psi$ is defined as in \lemref{convex}. Now by an application of Jensen's inequality, using \lemref{convex}, we have that
    \begin{align*}
        f_n(t) \leq nN -\int_0^t\Psi( E[X_{v_1}(s)], \ldots, E[X_{v_N}(s)])\;ds \leq nN -\int_0^t \Omega(f_n(s))\;ds,
    \end{align*}

    \noindent where the last inequality is by the definition of $\Omega$. Now by using the ODE comparison theorem we have that $f_n(t) \leq w_n(t)$ as desired.
\end{proof}

\noindent We are now equipped to prove \propref{comedown} using the above lemma, an argument that is similar to the proof of the reverse direction to Theorem 1.8 in \cite{Johnston2023}.

\begin{proof}[Proof of \propref{comedown}]
    The proof of this proposition is essentially the same as the proof of the reverse direction of Theorem 1.8 in \cite{Johnston2023}. We shall first note that since the $\Lambda_{v_i}$-coalescent comes down from infinity for all $i = 1, \ldots, N$ we have that
    \begin{align*}
        \int_2^\infty \frac{1}{\psi_{v_i}(q)}\;dq < \infty.
    \end{align*}

    \noindent Thus by an application of \lemref{upperboundminrate}, for all $s \geq 2N$ we have that
    \begin{align*}
        \int_{s}^\infty \frac{1}{\Omega(q)}\;dq < \infty.
    \end{align*}

    \noindent Thus for $t$ sufficiently small independent of $n$, we can see that the $w_n(t)$ defined in \lemref{odecomp} can be equivalently be defined as
    \begin{align*}
        t = \int_{w_n(t)}^{nN} \frac{1}{\Omega(q)}\;dq.
    \end{align*}

    \noindent Thus for $t$ sufficiently small we can now define $w_\infty(t)$ to be the solution to the integral equation
    \begin{align*}
        t = \int_{w_\infty(t)}^\infty \frac{1}{\Omega(q)}\;dq.
    \end{align*}

    \noindent Now it is clear that $w_n(t) \leq w_{n+1}(t) \leq \cdots \leq w_\infty(t)$. Thus by an application of \lemref{odecomp} we have that for all $n \in \N$ that
    \begin{align*}
        E[|X_n(t)|] \leq w_\infty(t) < \infty.
    \end{align*}

    \noindent Now taking $n \to \infty$ we have $|X_n(t)| \uparrow |X(t)|$ a.s. by part (b) of \thmref{PPPconstruction}, so by the monotone convergence theorem we have $|X(t)| < \infty$ almost surely for sufficiently small $t$. This proves that the coordinated particle system comes down from infinity.
\end{proof}

\noindent Now using this proposition we are able to give a quick proof of \thmref{converse}.

\begin{proof}[Proof of \thmref{converse}]
    We shall prove the contrapositive to this statement. It is clear that if condition (a) doesn't hold then the coordinated particle system starts with only finitely many particles which means that it comes down from infinity. 
    \medskip
    
    \noindent Now assume that condition (b) does not hold, i.e. for every site $v \in V$ the measure $\Lambda_v(dz)$ is such that the $\Lambda_v$-coalescent comes down from infinity. Now as there is no reproduction, and death can only decrease the number of particles, an application of \propref{comedown} shows us that the coordinated particle system $(X(t))_{t \geq 0}$ comes down from infinity in this case. 
    \medskip
    
    \noindent Now assume that condition (c) does not hold. Then we can see that there exists a subset of sites $V' \subseteq V$ such that for all $v \in V'$ the $\Lambda_v$-coalescent comes down from infinity, all the sites in $V \setminus V'$ start with only finitely many particles, and all the migration and reproduction measures $M_{uv}(dz)$ and $R_{uv}(dz)$ for $u \in V'$ and $v \in V\setminus V'$ are not $\Lambda$-strong. This means that only finitely many particles will enter sites in $V \setminus V'$. Thus we can reduce to the case where condition (b) does not hold by just considering the sites $V'$ and thus see that $(X(t))_{t \geq 0}$ comes down from infinity.
\end{proof}

\section*{Acknowledgements}

The author would like to thank Professor Jason Schweinsberg for his support and guidance as well as his detailed comments on earlier drafts of this work.

\bibliographystyle{alpha}
\bibliography{bibliography}

\end{document}